\documentclass[reqno,12pt,draft]{amsart}
\usepackage{amsmath,amsthm,amsfonts,amssymb}
\usepackage[cp1251]{inputenc}

\date{}
\newtheorem{theorem}{Theorem}[section]
\newtheorem{lemma}[theorem]{Lemma}

\newtheorem{remark}[theorem]{Remark}
\newtheorem{proposition}[theorem]{Proposition}

\newtheorem{example}[theorem]{Example}
\numberwithin{equation}{section}

\begin{document}

\centerline{\sc On uniqueness of solutions to nonlinear Fokker--Planck--Kolmogorov equations}

\vskip 2 ex

{\sc Oxana A. Manita ${}^{a}$ \footnote{Corresponding author
\par
e-mails: oxana.manita@gmail.com (O.A.Manita), mcliz@mail.ru (M.S.Romanov), starticle@mail.ru (S.V.Shaposhnikov).},
Maxim S. Romanov${}^{a}$, Stanislav V. Shaposhnikov${}^{a}$}

\vskip 1. ex

\quad ${}^{a}$ Department of Mechanics and Mathematics, Moscow State
University, 119991, GSP-1, 1 Leninskie Gory, Moscow, Russia;

\vskip 4 ex
We study uniqueness of flows of probability measures solving the Cauchy problem
for nonlinear Fokker--Planck--Kolmogorov
equation with unbounded coefficients. Sufficient conditions for uniqueness are indicated and
examples of non-uniqueness are constructed.

\vskip 2 ex


\vskip 2 ex


\vskip 2 ex

\section{\sc Introduction.}

In this paper we study  uniqueness of solutions to the following Cauchy problem
for a nonlinear Fokker--Planck--Kolmogorov equation:
\begin{equation}\label{e1}
\partial_t\mu=\partial_{x_i}\partial_{x_j}(a^{ij}(\mu, x, t)\mu)-\partial_{x_i}(b^i(\mu, x, t)\mu),
\quad \mu|_{t=0}=\nu.
\end{equation}
A solution is a  finite Borel measure $\mu(dxdt)=\mu_t(dx)\,dt$ given by a flow of
probability measures $(\mu_t)_{t\in[0, T]}$ on $\mathbb{R}^d$.
The equation is understood in the sense of distributions. Precise definitions are given below.
Throughout the paper we assume that the diffusion matrix  $A=(a^{ij})$ is symmetric and non-negative definite.

The main goal of this work is to establish sufficient conditions for uniqueness that
allow nonsmooth and unbounded coefficients, for instance,  coefficients given by convolutions with
kernels rapidly growing at infinity. Moreover, we investigate more difficult cases
where the diffusion matrix is degenerate or depends on a solution.
Finally, we construct several examples of nonuniqueness.

Equations of this form, called Fokker--Planck--Kolmogorov equations, generalize several types of equations
important for applications: transport equations, Vlasov equations, linear Fokker--Planck--Kolmogorov equations,
and McKean--Vlasov  equations.
Such equations describe the evolution of the initial measure $\nu$
under the action of a flow generated by a system of ODEs or SDEs.
An extensive literature is devoted to each type of equations.
Let us mention the classical paper by  Kolmogorov \cite{K},
where he derived linear Fokker--Planck--Kolmogorov equations for the transition probabilities of
 diffusion processes and the papers by McKean \cite{MM}, \cite{M} concerned with nonlinear
parabolic equations.
In the general case, such equations and the well-posedness of the martingale problem
were studied by Funaki~\cite{Fun}.
In particular, he  obtained the following uniqueness result.
Let us consider coefficients of the form  $a^{ij}(x, \mu_t)$ and $b^i(x, \mu_t)$.
If the corresponding martingale problem and the corresponding linear equation have unique solutions,
then the Cauchy problem~(\ref{e1}) has a unique solution. Uniqueness for the martingale problem
was established under the following assumption:
$$
|\sqrt{A}(x, \mu_t)-\sqrt{A}(y, \sigma_t)|+|b(x, \mu_t)-b(y, \sigma_t)|\le C|x-y|+G(w_p(\mu_t, \sigma_t))
$$
where  $A=(a^{ij})$, $w_p$ is the Kantorovich  $p$-metric and $G$
is an increasing continuous function on $[0, +\infty)$ with $G(0)=0$ and
$\displaystyle\int_{0+}G^{-2}(\sqrt{u})\,du=+\infty$.
Thus, only globally Lipschitz coefficients are  admitted. Moreover, the dependence of the coefficients on
$\mu$ allows in fact only convolutions with polynomially growing kernels.
Let us emphasize also a rather conditional character of the uniqueness statement. Indeed,
it requires to have a priori the  uniqueness for the corresponding linear problem and the martingale problem.
In the one-dimensional case, the uniqueness of the martingale problem
(or equivalently, the uniqueness of a weak
solution to the corresponding McKean--Vlasov SDE) is studied in
\cite{B1}, \cite{B2}, where the diffusion matrix is assumed to be unit
and the drift is a convolution with an odd and monotone kernel.
Strong solutions to McKean--Vlasov SDEs are studied in \cite{V}.
Some examples of nonuniqueness for equations of the same
type with the identically zero diffusion matrix are constructed in \cite{S}.
Vlasov equations with smooth coefficients were by Dobrushin \cite{D},
who proved some   existence and uniqueness theorems employing the
 contraction mapping principle with a special choice
of a probability metric. Surveys of recent results on  Vlasov equations are given in  \cite{VVK1}, \cite{VVK2}.
Transport equations, linear  Fokker--Planck--Kolmogorov equations, Vlasov equations and Boltzmann equations
with Sobolev coefficients are investigated in  \cite{DL1}, \cite{DL2}, and \cite{LBL}, where
the method of renormalized solutions is developed and  existence and uniqueness problems are
studied in the space~$L^p$. Many papers (see, e.g., \cite{AS}, \cite{AF}, and \cite{JOT})
dealing with the unit diffusion matrix and drifts of the form
$$
b(x, \mu_t)=\nabla\psi(x)+\int \nabla W(x-y)\,d\mu_t
$$
develop the gradient flow approach. Diverse physical problems leading to the study of nonlinear
Fokker--Planck--Kolmogorov equations can be found in  \cite{DFP}.
Despite the vast literature on the topic,
there are almost no general results concerning uniqueness
in the cases of non-Lipschitz and rapidly growing coefficients.
In this general formulation the existence of the solution has been
investigated in \cite{BRSH-n} \cite{MSHD}, and \cite{MSH}.
For surveys of results concerning existence and
uniqueness in the linear case, see \cite{BDR}, \cite{BKR-s}, and \cite{BRSH11}.

In the present work to prove uniqueness we use a modification of the classical Holmgren method,
which can be illustrated as follows. Suppose there are two solutions  $\mu$ and $\sigma$.
We solve the adjoint problems
$$
\partial_tf+a^{ij}(\mu)\partial_{x_i}\partial_{x_j}f+b^i(\mu)\partial_{x_i}f=0, \quad f|_{s=t}=\psi,
$$
where $\psi\in C^{\infty}_0(\mathbb{R}^d)$, in the class of sufficiently smooth functions.
Multyplying by  $f$ the equation (\ref{e1}) and integrating by parts, we come to
$$
\int\psi\,d(\mu_t-\sigma_t)=\int_0^t\int(L_{\mu}-L_{\sigma})f\,d\sigma_s\,ds.
$$
Now let us choose a metric on the space of probability measures, for example, Kantorovich  1-metric
$$
W_1(\mu_t, \sigma_t)=\sup\Biggl\{\int\psi\,d(\mu_t-\sigma_t): |\nabla\psi|\le 1\Biggr\}
$$
on the subset of probability measures having finite first moments. Now we estimate the right-hand side
with it. If  $a^{ij}$ is independent of the solution, then the right-hand side has the form
$$
\int_0^t\int\langle b(\mu)-b(\sigma), \nabla f\rangle\,d\sigma_s\,ds.
$$
Suppose  $|b(\mu)-b(\sigma)|\le CW_1(\mu_t, \sigma_t)$, and  $|\nabla f|$ due to the maximum principle is
dominated by  $\max|\nabla\psi|$. Since  $\psi$ is arbitrary, we come to
$$
W_1(\mu_t, \sigma_t)\le C\int_0^tW_1(\mu_s, \sigma_s)\,ds,
$$
and Gronwall's inequality yields  $W_1(\mu_t, \sigma_t)=0$.

One of the main difficulties of this approach is solving the adjoint problem with nonregular
and unbounded coefficients. To evade this difficulty, we approximate the operator~$L$
with a sequence of operators with smooth coefficients and solve the adjoint problem for
them. Other difficult task is to choose a metric on the space of measures. This choice is determined by
assumptions on $\psi$, which are, in their turn, determined by apiori estimates for~$f$.
In the present paper we consider three different situations: the
diffusion matrix $A$ is non-degenerate and independent of the solution,
the diffusion matrix $A$ is degenerate and independent of the solution, the
diffusion matrix $A$ depends on the solution. In the first case we choose a weighted total variation metric.
 This choice is partially motivated by the fact that in this case solutions
 have densities with respect to Lebesgue measure
 and it is natural to consider weighted~$L^1$ spaces.
 In the second situation we use the generalisation of Fortet-Mourier metric.
 Since this metric is different from standard ones~(cf. \cite{Bogachev-2},\cite{BKol},\cite{Rachev91}),
 we also study the relation between the new metric and Kantorovich $p$ - metric
 and classical Fortet-Mourier metric.

Let us give the precise definitions. Recall that a measure $\mu$ on $\mathbb{R}^d\times[0, T]$ is given by a flow of
probability measures  $(\mu_t)_{t\in[0, T]}$ on $\mathbb{R}^d$  if
$\mu_t\ge 0$, $\mu_t(\mathbb{R}^d)=1$, for each Borel set $B$ the function $t\mapsto \mu_t(B)$
 is measurable and
$$
\int_0^T\int u\,d\mu=\int_0^T\int u\,d\mu_t\,dt \quad \forall u\in C^{\infty}_0(\mathbb{R}^d\times(0, T)).
$$
For shortness further we write  $\mu(dxdt)=\mu_t(dx)\,dt$. Set
$$
L_{\mu}u=a^{ij}(\mu, x, t)\partial_{x_i}\partial_{x_j}u+b^i(\mu, x, t)\partial_{x_i}u.
$$

We shall say that  $\mu(dxdt)=\mu_t(dx)\,dt$ satisfies the Cauchy problem  (\ref{e1}) if
 we have mappings $(x, t)\mapsto a^{ij}(\mu, x, t)$,
$(x, t)\mapsto b^{i}(\mu, x, t)$ and  $a^{ij}, b^i\in L^1(\mu, U\times[0, T])$ for each ball $U\subset\mathbb{R}^d$
and for each function  $\varphi\in C^{\infty}_0(\mathbb{R}^d)$
 the following identity holds
\begin{equation}\label{r1}
\int\varphi\,d\mu_t=\int\varphi\,d\nu+\int_0^t\int L_{\mu}\varphi\,d\mu_s\,ds
\end{equation}
for all $t\in[0, T]$.
Sometimes it is more convenient to use an equivalent definition that requires  (instead of (\ref{r1}))
the identity
\begin{equation}\label{r2}
\int u(x, t)\,d\mu_t=\int u(x, 0)\,d\nu+\int_0^t\int\bigl[\partial_tu+L_{\mu}u\bigr]\,d\mu_s\,ds
\end{equation}
for all $t\in[0, T]$ for each test function $u\in C^{1, 2}(\mathbb{R}^d\times(0, T))\bigcap C(\mathbb{R}^d\times[0, T))$
that equals zero outside some ball $B\subset\mathbb{R}^d$.
In particular, the flow of probability measures  $\mu_t$ satisfying the Cauchy problem is continuous in  $t$
with respect to the weak convergence of probability measures. This follows directly from the continuity in $t$
of the integrals  $\displaystyle\int\varphi\,d\mu_t$ for each function  $\varphi\in C^{\infty}_0(\mathbb{R}^d)$;
the latter is ensured by the identity  (\ref{r1}).

Since we admit unbounded coefficients and convolutions with unbounded kernels,
we consider measures that integrate some function, growing at infinity.
It will be explained that this "apriori integrability" can be ensured by an appropriate
Lyapunov function. So, we consider solution from the class  $M_T(V)$ of measures  $\mu$ on $\mathbb{R}^d\times[0, T]$
given by flows of probability measures  $(\mu_t)_{t\in[0, T]}$ and
satisfying
\begin{equation}\label{est}
\sup_{t\in[0, T]}\int V(x)\,d\mu_t<\infty,
\end{equation}
where  $V\ge 1$ and, generally speaking, $V$  unbounded as  $|x|\to\infty$.

Remind sufficient conditions for {\it existence} of solutions, established in~\cite{MSH}.
Set~$\tau_0>0$.  $C^{+}([0, \tau_0])$ denotes the set of nonnegative continuous functions on
$[0, \tau_0]$. For each function $\alpha\in C^{+}([0, \tau_0])$ and each  $\tau\in(0, \tau_0)$
let $M_{\tau, \alpha}(V)$ denote the set of measures  $\mu$ given by  flows of probability measures
$(\mu_t)_{t\in[0, \tau]}$ satisfying
$$
\int V(x)\,d\mu_t\le\alpha(t) \quad \forall \, t\in[0, \tau].
$$

{\it First condition:} there is a function  $V\in C^{2}(\mathbb{R}^{d})$,
$V(x)>0$, $\lim_{|x|\to+\infty}V(x)=+\infty$
and mappings  $\Lambda_{1}$ and $\Lambda_{2}$ of the space  $C^{+}([0,\tau_{0}])$
to $C^{+}([0,\tau_{0}])$ such that for all $\tau\in(0,\tau_{0}]$
and $\alpha\in C^{+}([0,\tau_{0}])$ functions $a^{ij}$
and $b^{i}$ are defined on  $M_{\tau,\alpha}=M_{\tau,\alpha}(V)$
and for all  $\mu\in M_{\tau,\alpha}$ and all $(x,t)\in\mathbb{R}^{d}\times[0,\tau]$
one has
$$
L_{\mu}V(x,t)\leq\Lambda_{1}[\alpha](t)+\Lambda_{2}[\alpha](t)V(x).
$$
We shall call such function  $V$ a Lyapunov function for the operator  $L_{\mu}$.

{\it Second condition:} for all $\tau\in(0,\tau_{0}]$, $\alpha\in C^{+}([0,\tau_{0}])$,
$\sigma\in M_{\tau,\alpha}$ and $x\in\mathbb{R}^{d}$ the mappings
$$
t\mapsto a^{ij}(x,t,\sigma)\quad\hbox{and}\quad t\mapsto b^{i}(x,t,\sigma)
$$
are Borel measurable on  $[0,\tau]$ and for each closed ball  $U\subset\mathbb{R}^{d}$
the mappings
$$
x\mapsto b^{i}(x,t,\sigma)\quad\hbox{and}\quad x\mapsto a^{ij}(x,t,\sigma)
$$
are bounded on  $U$ uniformly in  $\sigma\in M_{\tau,\alpha}$ and $t\in[0,\tau]$
 and continuous on  $U$ uniformly in  $\sigma\in M_{\tau,\alpha}$
and $t\in[0,\tau]$. Moreover, if a sequence  $\mu^{n}\in M_{\tau,\alpha}$
$V$-converges to  $\mu\in M_{\tau,\alpha}$, i.e. by definition
 for each function  $F\in C(\mathbb{R}^d)$ such that  $\lim_{|x|\to\infty}F(x)/V(x)=0$
 one has
$$\lim_{n\to\infty}\int F\,d\mu^n_t=\int F\,d\mu_t$$
for each $t\in[0, \tau]$,
 then for all  $(x,t)\in\mathbb{R}^{d}\times[0,\tau]$
one has
$$
\lim_{n\rightarrow\infty}a^{ij}(x,t,\mu^{n})=a^{ij}(x,t,\mu),
\quad\quad\lim_{n\rightarrow\infty}b^{i}(x,t,\mu^{n})=b^{i}(x,t,\mu).
$$

So, if this two conditions are fulfilled, there exists  $\tau\in(0, \tau_0]$ such that on the interval $[0, \tau]$
 there exists a solution  $\mu$ to the Cauchy problem  (\ref{e1}) and  $\mu$ is given by a flow of probability measures $\mu_t$
 satisfying  (\ref{est}) with $\tau$ instead of $T$.

In the present paper we use three different Lyapunov functions: the function $V$ to define the
class $M_T(V)$ in which we  solve our problem,
the function $W$ to determine the dependence of the coefficients on the solution,
the function $U$ to control the growth of the coefficients at infinity.

We point out that the method of Lyapunov functions for equations of this type
was introduced by Hasminskii in \, \cite{H}, and was recently developped in the study of linear \,
Fokker--Planck--Kolmogorov equations with unbounded coefficients~(for instance, cf.\cite{BR},\cite{BDR},\cite{BRSH11}).
For further consideration the following fact from \cite{BDR} (also cf. \cite{SH-MO}) is important.
If a measure  $\mu$ given by a flow of probability measures~$\mu_t$ satisfies the Cauchy problem~(\ref{e1})
and for some function  $V\in C^2(\mathbb{R}^d)$ such that  $\lim_{|x|\to\infty} V(x)=+\infty$ and $V\in L^1(\nu)$
there is a number  $C$ such that  $L_{\mu}V\le C+CV$, then for a.e.  $t\in[0, T]$ one has
$$
\int V(x)\,d\mu_t\le e^{Ct}+e^{Ct}\int V(x)\,d\nu.
$$
Moreover, if  $|\sqrt{A}\nabla V|\le \widetilde{C}V$ for some  $\widetilde{C}$,
then the latter estimate holds for  $V^m$ with any number $m\ge 1$.

The present paper consists of six sections. The first section is the introduction,
the second section contains an approximation lemma, the third and
the fourth deal with  non-degenerate and degenerate diffusion matrix independent of the solution. The fifth section
concerns the case of the diffusion matrix, depending on the solution, the sixth contains some examples of nonuniqueness.

\section{\sc Approximation lemma}

It is well-known that localy integrable or  bounded functions admit good approximations by convolutions with smooth kernels.
However we need to control the existence of the Lyapunov function for these approximations.

\begin{lemma}\label{approx}
 Suppose  $a^{ij}$, $b^i$ are Borel functions on  $\mathbb{R}^{d+1}$, bounded on $B\times[\alpha, \beta]$
for each ball  $B\subset\mathbb{R}^d$ and each interval $[\alpha, \beta]$.
Suppose there exist functions  $W\in C^2(\mathbb{R}^d)$ and $\Lambda\in C(\mathbb{R}^d)$ such that
$W\ge 1$ and
$$
a^{ij}(x, t)\partial_{x_ix_j}W(x)+b^i(x, t)\partial_{x_i}W(x)\le \Lambda(x)W(x)
\quad \forall (x, t)\in \mathbb{R}^{d+1}.
$$
Then the following assertions hold:

\, {\rm (i)} \, there exist sequences of functions  $a^{ij}_m, b^i_n\in C^{\infty}(\mathbb{R}^{d+1})$ such that
 for each measure  $\mu=\varrho(x, t)\,dx\,dt$, where  $\varrho$ is a Borel nonnegative function
 and
$\|\varrho(\,\cdot\,, t)\|_{L^1(\mathbb{R}^d)}=1$ for~a.~e.~$t$, one has
$$
\lim_{m\to\infty}\|a^{ij}_m-a^{ij}\|_{L^p(\mu, B\times[\alpha, \beta])}=0, \quad
\lim_{n\to\infty}\|b^{i}_n-b^{i}\|_{L^p(\mu, B\times[\alpha, \beta])}=0
$$
for each  $p\ge 1$, each ball $B\subset\mathbb{R}^d$ and each interval $[\alpha, \beta]$.

\, {\rm (ii)} \, Suppose  $a^{ij}$, $b^i$ are continuous in $x$ uniformly in~$t$
on $B\times[\alpha, \beta]$ for each ball  $B\subset\mathbb{R}^d$ and each interval $[\alpha, \beta]$.
Suppose $\mu$ is a Borel measure on $\mathbb{R}^{d+1}$ given by a flow of probability measures $\mu_t$ on $\mathbb{R}^d$,
i.e. $\mu(dxdt)=\mu_t(dx)\,dt$.
Then there exist sequences of functions $a^{ij}_m, b^i_n\in C^{\infty}(\mathbb{R}^{d+1})$ such that
$$
\lim_{m\to\infty}\|a^{ij}_m-a^{ij}\|_{L^p(\mu, B\times[\alpha, \beta])}=0, \quad
\lim_{n\to\infty}\|b^{i}_n-b^{i}\|_{L^p(\mu, B\times[\alpha, \beta])}=0
$$
for each  $p\ge 1$, each ball  $B\subset\mathbb{R}^d$ and each interval  $[\alpha, \beta]$.

\, {\rm (iii)} \, In  {\rm (i)} and {\rm (ii)} for each ball  $B\subset\mathbb{R}^d$ and each interval $[\alpha, \beta]$
one can find an index~$n_0$ such that  for all $m, n>n_0$
one has
$$
a^{ij}_m(x, t)\partial_{x_ix_j}W(x)+b^i_n(x, t)\partial_{x_i}W(x)\le (1+\Lambda(x))W(x)
\quad \forall (x, t)\in B\times[\alpha, \beta].
$$
\end{lemma}
\begin{proof}
Let  $\xi\in C^{\infty}_0(\mathbb{R}^d)$ and $\eta\in C^{\infty}_0(\mathbb{R})$
 be  smoothing kernels, i.e. $\xi\ge 0$, $\|\xi\|_{L^1(\mathbb{R}^d)}=1$ and
 $\eta\ge 0$, $\|\eta\|_{L^1(\mathbb{R}^1)}=1$.
For each  $\varepsilon>0$ set
$$
\xi_{\varepsilon}(x)=\varepsilon^{-d}\xi(x/\varepsilon),
\quad
\eta_{\varepsilon}(t)=\varepsilon^{-1}\eta(t/\varepsilon),
\quad
\omega_{\varepsilon}(x, t)=\xi_{\varepsilon}(x)\eta_{\varepsilon}(t).
$$
Let us prove (i). Sequences  $a^{ij}_m=\omega_{1/m}*a^{ij}$ and $b^i_n=\omega_{1/n}*b^i$
converge to  $a^{ij}$ and $b^i$ for~a.~e.~$(x, t)$ and are bounded on each set $B\times[\alpha, \beta]$
where $B$ is a ball.
Taking into account that $\|\varrho(\,\cdot\,, t)\|_{L^1(\mathbb{R}^d)}=1$ and using Lebesgue's dominated theorem, one gets the required
assertion. Let us check (iii) in this case.
Suppose  $g$ is a Borel function on $\mathbb{R}^{d+1}$, bounded on $B\times[\alpha, \beta]$
for each ball  $B$ and interval  $[\alpha, \beta]$. Suppose also that there exist functions
$\varphi, \psi\in C(\mathbb{R}^d)$ satisfying  $\varphi(x)g(x, t)\le \psi(x)$
for all $(x, t)\in\mathbb{R}^d$. As above,  $g_n=\omega_{1/n}*g$. To check  (iii)
it suffices to prove that for each ball  $B\subset\mathbb{R}^d$ and interval  $[\alpha, \beta]$
there is an index  $n_0$ such that for each $n>n_0$
one has  $\varphi(x)g_n(x, t)\le \psi(x)+1$ for all $(x, t)\in B\times[\alpha, \beta]$.
Indeed,
$$
\varphi(x)g_n(x, t)\le \psi(x)+
\int\int\Biggl(\bigl(\psi(y)-\psi(x)\bigr)+\bigl(\varphi(x)-\varphi(y)\bigr)g(y, \tau)\Biggr)
\omega_{1/n}(x-y, t-\tau)\,dy\,d\tau.
$$
The assertion follows from the continuity  $\varphi, \psi$ and the fact that $g$ is bounded.

Let us prove  (ii). It suffices to construct a sequence of continuous functions approximating
$a^{ij}$ and $b^i$ since continuous functions admit a uniform approximation by smooth functions.
For each  $x$ set $a^{ij}_m(x, t)=a^{ij}(x, \,\cdot\,)*\eta_{1/m}(t)$ and
$b^{i}_n(x, t)=b^{i}(x, \,\cdot\,)*\eta_{1/n}(t)$. Note that the uniform continuity of  $a^{ij}$ and $b^i$
 in  $x$  yields the continuity of  $a^{ij}_m$, $b^i_n$ in the pair of variables.
Moreover, since  $\Lambda$ and $W$ are independent of  $t$, inequality from (iii)
is obviously fulfilled for   $a^{ij}_m$, $b^i_n$. Due to the properties of convolutions for each  $x$
sequences  $a^{ij}_m(x, t)$ and $b^i_n(x, t)$ converge to
$a^{ij}_m(x, t)$ and $b^i_n(x, t)$ for a.e. $t$. Again using the uniform continuity in $x$ we derive
the existence of a set $J\subset[\alpha, \beta]$ of full Lebesgue measure such that the convergence takes place for
all $(x, t)\in\mathbb{R}^d\times J$. Lebesgue's dominated theorem ensures
$$
\lim_{m\to\infty}\int_B|a^{ij}_n(x, t)-a^{ij}(x, t)|^p\,d\mu_t=0, \quad
\lim_{n\to\infty}\int_B|b^i_n(x, t)-b^i(x, t)|^p\,d\mu_t=0
$$
for a.e. $t\in J$. Boundness of  $a^{ij}, b^i$ and the fact $\mu_t$ are probability measures
yield the required assertion.
\end{proof}

\begin{remark}\label{approx-r}\rm
\, (i) \, If the coefficients  $a^{ij}$ are continuous in  $(x, t)$,
 the assertion (iii) of Lemma stays true if one replaces  $a^{ij}_m$ with  $a^{ij}$ in the inequality. Indeed,
 since  $a^{ij}_m$ are constructed by convolutions with smooth kernels, they converge uniformly to  $a^{ij}$
 on each compact set.

\, (ii) \, From the proof one can see that if $|b^i(x, t)|\le \varphi(x)$
for some continuous function~$\varphi$ and
all $(x, t)\in\mathbb{R}^{d+1}$, then
 for each ball  $B\subset\mathbb{R}^d$ and interval $[\alpha, \beta]$
there is an index  $n_0$ such that for each $n>n_0$
one has $|b^i_n(x, t)|\le \varphi(x)+1$ for $(x, t)\in B\times[\alpha, \beta]$.

\,(iii) \, If
$\langle b(x+y, t)-b(x, t), y\rangle\le \theta(x)|y|^2$
for all  $x, y, t$ and some continuous function~$\theta$,
then for each ball  $B\subset\mathbb{R}^d$ and segment  $[\alpha, \beta]$
 there is an index  $n_0$ such that for each  $n>n_0$ one has
$\langle b_n(x+y, t)-b_n(x, t), y\rangle\le (\theta(x)+1)|y|^2$
for all~$(x, t)\in B\times[\alpha, \beta]$
and all~$y\in \mathbb{R}^d$.

\, (iv) \, If  $\lambda|\xi|^2\le \langle A(x, t)\xi, \xi\rangle \le\lambda^{-1}|\xi|^2$
for all  $x, y\in\mathbb{R}^d$ and $t\in\mathbb{R}$, then the same inequalities with the same constant $\lambda$
 hold for $A_m$ this follows from properties of the convolution and the kernel  $\omega_{\varepsilon}$.
Moreover, if  $A$ is Lipschitz or H\"older in  $x$ with the Lipschitz constant $\Lambda$, then  $A_m$ is Lipschitz in $x$ with
the Lipschitz constant $\Lambda$.
\end{remark}

\section{\sc Diffusion matrix is non-degenerate and is independent of  $\mu$.}

In this section we study the case when coefficients  $a^{ij}$ are independent of  $\mu$ and ${\rm det}A>0$.
So, suppose the following assumption holds.

\, (H1)\, There exists a continuous positive function  $\lambda$ on  $\mathbb{R}^d$ such that
$$\langle A(x, t)\xi, \xi\rangle\ge \lambda(x)|\xi|^2$$
 and all $(x, t)\in\mathbb{R}^d\times[0, T]$ and $\xi\in\mathbb{R}^d$,
and for each ball  $B$ there exist such numbers  $\gamma=\gamma(B)>0$ and $\kappa=\kappa(B)\in(0, 1]$ that
$$|a^{ij}(x, t)-a^{ij}(y, t)|\le \gamma|x-y|^{\kappa}$$
for all $x, y\in B$, $t\in[0, T]$.

Let  $\|\mu\|$ denote the total variation of the measure $\mu$. Note that if the measure is given
by a density  $\varrho$
with respect to Lebesgue measure then its total variation is equal to  $L^1$-norm of its density.
 Set  $\|\mu\|_W=\|W\mu\|$ for each measurable positive function~$W$.

Suppose a continuous function $V\ge 1$ is given. As above,  $\mathcal{M}_{T}(V)$
 denotes the set of such measures  $\mu$ on $\mathbb{R}^d\times[0, T]$ that
$\mu$ is given by a flow of probability measures  $\mu_t$ on $\mathbb{R}^d$ and
$$
\sup_{t\in[0, T]}\int V(x)\,d\mu_t<\infty.
$$

In addition to (H1) we assume the following conditions:

\, (H2) \, there exists a function  $W\in C^2(\mathbb{R}^d)$,
$W>0$, $\lim\limits_{|x|\to+\infty}W(x)=+\infty$ such that  $W(x)V^{-1/2}(x)$ is bounded on  $\mathbb{R}^d$
and for each  $\mu\in\mathcal{M}_{T}(V)$ there is a constant
$\alpha(\mu)>0$ such that
$$L_{\mu}W(x, t)\le \alpha(\mu)W(x)$$
for all $(x, t)\in \mathbb{R}^d\times[0, T]$;

\, (H3) \,
there exists a continuous increasing function  $G$ on $[0, +\infty)$ that  $G(0)=0$ and
$$
\lambda(x)^{-1}\bigl|b(\mu, x, t)-b(\sigma, x, t)\bigr|\le \sqrt{V(x)}G(\|\mu_t-\sigma_t\|_{W})
$$
for all  $(x, t)\in\mathbb{R}^d\times[0, T]$ and $\mu, \sigma\in \mathcal{M}_{T}(V)$;

\, (H4) \, there exists a function  $U\in C^2(\mathbb{R}^d)$,
$U>0$, $\lim_{|x|\to\infty}U(x)=+\infty$ such that for each $\mu\in \mathcal{M}_{T}(V)$
there exists a number  $\beta(\mu)>0$ such
that
$$
W^2(x)\lambda(x)^{-1}|b(\mu, x, t)|^2
+\frac{|\sqrt{A(x, t)}\nabla U(x)|^2}{U^2(x)}
+\frac{|L_{\mu}U(x, t)|}{U(x)}\le \beta(\mu)V(x)
$$
for all  $(x, t)\in\mathbb{R}^d\times[0, T]$.

\begin{theorem}\label{th1}
Suppose  {\rm (H1)}, {\rm (H2)}, {\rm (H3)}, {\rm (H4)} hold true.
If
$$
\int_{0+}\frac{du}{G^2(\sqrt{u})}=+\infty,
$$
then there exists at most one solution to the Cauchy problem {\rm (\ref{e1})} from the class
$\mathcal{M}_{T}(V)$.
\end{theorem}

\begin{example}\rm
Let  $m\ge k\ge 1$ and for each measure  $\mu=(\mu_t)$ that
\begin{equation}\label{con1}
\sup_{t\in[0, T]}\int |x|^{2m}\,d\mu_t<\infty
\end{equation}
there exist constants $c_1(\mu)>0$, $c_2(\mu)>0$ such that
$$
\langle b(\mu, x, t), x\rangle\le c_1(\mu)(1+|x|^2), \quad |b(\mu, x, t)|\le c_2(\mu)(1+|x|^{m-k}).
$$
Suppose there exists such a number $c_3>0$ that
$$
|b(\mu, x, t)-b(\sigma, x, t)|\le c_3(1+|x|^m)\int(1+|y|^k)\,d|\mu_t-\sigma_t|
$$
for all  $\mu, \sigma$ satisfying  (\ref{con1}).

Then the  Cauchy problem
$$
\partial_t\mu=\Delta\mu-{\rm div}(b(\mu, x, t)\mu), \quad \mu|_{t=0}=\nu,
$$
has at most one solution satisfying  (\ref{con1}).

In particular, all assumptions are fulfilled for
$$
b(\mu, x, t)=-\int|x-y|^{n}(x-y)\mu_t(dy)
$$
 with  $m=2n+2$ and $k=n+1$.

To prove this fact it is sufficient to apply Theorem for $A=I$, $V(x)=1+|x|^{2m}$, $W(x)=1+|x|^{k}$ and $U(x)=1+|x|^2$.
\end{example}

Now we proceed to the proof of the Theorem.

\begin{proof}
Suppose  $\mu(dxdt)=\mu_t(dx)\,dt$ and $\sigma(dxdt)=\sigma_t(dx)\,dt$
are two solutions to the Cauchy problem~(\ref{e1}).
Set  $\alpha=\max\{\alpha(\mu), \alpha(\sigma)\}$,
$\beta=\sup\{\beta(\mu), \beta(\sigma), W(x)V^{-1/2}(x)\}$ and
$$
M=\sup_{t\in[0, T]}\int V(x)\,d(\mu_t+\sigma_t).
$$
Further we assume that conditions (H2) and (H4) are fulfilled with  $\alpha$ and $\beta$ indicated above.
Since  $A$ is non-degenerate, measures $\mu$ and $\sigma$ are given by densities $\varrho_{\mu}$
and $\varrho_{\sigma}$ with respect to Lebesgue measure (cf. \cite{BKR-s}),
and
$\|\varrho_{\mu}(\,\cdot\,, t)\|_{L^1(\mathbb{R}^d)}=1$,
$\|\varrho_{\sigma}(\,\cdot\,, t)\|_{L^1(\mathbb{R}^d)}=1$
for~a.~e.~$t$.
Hence we can apply statement (i) of Lemma \ref{approx}.

Let  $\varphi\in C^{\infty}_0(\mathbb{R})$ be a cut-off function such that  $0\le\varphi\le 1$, $\varphi(x)=1$ for $|x|<1$ and
$\varphi(x)=0$ for $|x|>2$. Assume also that for some  $C>0$ and
all $x\in\mathbb{R}$ one has
$|\varphi''(x)|^2+|\varphi'(x)|^2\le C\varphi(x)$.
For each  $N\ge 1$ set
$$
\varphi^W_N(x)=\varphi(W(x)/N) \quad \hbox{\rm and} \quad \varphi^U_N(x)=\varphi(U(x)/N),
$$
$$
B^W_N=\{x: \, W(x)\le N\} \quad \hbox{\rm and} \quad B^U_N=\{x: \, U(x)\le N\}.
$$

Suppose  $\psi\in C^{\infty}_0(\mathbb{R}^d)$ and $|\psi(x)|\le W(x)$. Fix such $K\ge 2$ that
 the support of  $\psi$ belongs to  $B_K^U$. Find such  $N=N(K)\ge 2$ that  $B_{2K}^U\subset B_N^W$ and fix the number $N(K)$.
 Note that  $\varphi^W_N(x)=1$ for $x\in{\rm supp}\,\varphi_N^U$.
The function $\varphi_N^U$ is used to localize the problem which permits to approximate
the coefficients of the operator  $L$ locally and not on the whole  $\mathbb{R}^d\times[0, T]$.
The function $\varphi_N^W$ cuts-off the coefficients in such a way that the new operator also
has a Lyapunov function, i.e.  (H2) is fulfilled.

Let us extend coefficients  $a^{ij}$, $b^i$ to $\mathbb{R}^{d+1}$ in the following way:
$a^{ij}(x, t)=a^{ij}(x, T)$, \, $b^i(x, t, \mu)=b^i(x, T, \mu)$ for $t>T$
and $a^{ij}(x, t)=a^{ij}(x, 0)$, \, $b^i(x, t, \mu)=b^i(x, 0, \mu)$ for $t<0$.
Obviously the extended coefficients satisfy
(H1), (H2), (H3) and (H4) on $\mathbb{R}^{d+1}$.

Now let us construct a new operator  $\widetilde{L}$ with smooth coefficients that approximates
 $L$ on  $B^W_{2N}\times[0, T]$.
Due to Lemma \ref{approx} there exist such sequences of functions
$b_n^i, a^{ij}_m\in C^{\infty}(\mathbb{R}^{d+1})$
that
$$
\lim_{m\to\infty}\|a^{ij}-a^{ij}_m\|_{L^{1}((\mu+\sigma), B^W_{2N}\times[0, T])}=0, \quad
\lim_{n\to\infty}\|b^{i}(\mu,\,\cdot\,,\,\cdot\,)-b^{i}_n\|_{L^{2}((\mu+\sigma), B^W_{2N}\times[0, T])}=0.
$$
According to Remark \ref{approx-r} the
matrix  $A_m=(a^{ij}_m)$ satisfies condition (H1) for each $m$ for all  $(x, t)\in\mathbb{R}^{d+1}$.
By Lemma \ref{approx} there exists an index  $n_0$ such that for all $m, n>n_0$ one has
$$
a^{ij}_m(x, t)\partial_{x_i}\partial_{x_j}W(x)+b^i_n(x, t)\partial_{x_i}W(x)\le(\alpha+1)W(x).
$$
Due to Remark \ref{approx-r} one has
$\lambda^{-1}(x)|b_n(x, t)|^2\le (\beta+1)V(x)$
for all $(x, t)\in B^W_{2N}\times[0, T]$.
Further  assume that $m, n>n_0$.

Set $\widetilde{A}=\varphi^W_{2N}A_m+(1-\varphi^W_{2N})I$, $\widetilde{b}=\varphi_{2N}b_n$ and
$\widetilde{L}=\widetilde{a}^{ij}\partial_{x_i}\partial_{x_j}+\widetilde{b}^i\partial_{x_i}$.

Now let us construct a Lyapunov function for $\widetilde{L}$ from $W$. We need it to estimate
maximum of the solution to the adjoint problem.
Let $W_N(x)=\zeta_K(W)$ where  $\zeta_N(z)=z$ for $z<N$, $\zeta(z)=N+1$ for $z>N+2$,
$0\le \zeta_{N}'\le 1$, $\zeta_{N}''\le 0$.
Note that  $W_N(x)\le W(x)$ and $\widetilde{L}W_N(x, t)\le (\alpha+1)W_N(x)$ for all
$(x, t)\in\mathbb{R}^{d+1}$. Indeed, $\widetilde{L}W_N=0$ outside $B^W_{N+2}$,
and on  $B^W_{N+2}$ one has
$$
\widetilde{L}W_N=\zeta_N'(W)\widetilde{L}W+\zeta''_N(W)|\sqrt{A_m}\nabla W|^2\le (\alpha+1)\zeta_N(W).
$$
Here the inequality $z\zeta_N'(z)\le \zeta_N(z)$ is used; it follows from
$(z\zeta_N'(z)-\zeta_N(z))'=z\zeta''_N(z)$ and the fact  $\zeta''_N(z)\le 0$.
This  $W_N$ is the required Lyapunov function.

Suppose  $s\in (0, T)$ and $f$ is a solution to the Cauchy problem
$\partial_t f+\widetilde{L}f=0$, $f|_{t=s}=\psi$. Since
 all coefficients are smooth and bounded together with all the derivatives,
 a smooth solution  $f$ exists and is bounded together with all the derivatives (for instance, cf. \cite{Fr}).
Function  $f$ depends on  $m$, $n$ and $N$,
but we omit the indeces for shortness.

Let us estimate $|f|$. Firstly we note that for fixed initial condition $\psi$
the maximum principle yields  $|f|\le \max|\psi|=C(\psi)$.
Now let us establish a bound independent of  $\psi$.
To do this, note that function $v=f/W_N$ satisfies
$$
\partial_tv+\widetilde{L}v+2\langle \widetilde{A}\nabla v, \nabla W_N\rangle W_N^{-1}+vW_N^{-1}\widetilde{L}W_N=0.
$$
According to the statement above,  $W_N^{-1}\widetilde{L}W_N\le \alpha+1$ and $|v(x, s)|=|\psi(x)|/W_N(x)\le 1$.
Maximum principle yields  $|v(x, t)|\le e^{(\alpha+1)(s-t)}$, which ensures
$$
|f(x, t)|\le W_N(x)e^{(\alpha+1)(s-t)}\le W(x)e^{(\alpha+1)(s-t)}.
$$
Note that by \cite[Theorem 2.8.]{KP} there exists a number  $C(n, N, \psi)$ that
$$
\sup_{(x, t)\in\mathbb{R}^d\times[0, s]}|\partial_{x_i}\partial_{x_j}f(x, t)|\le C(n, N, \psi).
$$
Further it will be important that  $C(n, N, \psi)$ is independent of  $m$.
Now let us estimate~$|\nabla_x f|$. Substituting a test function $u=\varphi_K^Uf^2$ into the identity (\ref{r2}) for
the solution $\mu$, we get
\begin{multline*}
\int \varphi_K^U f^2\,d\mu_s-\int \varphi_K^U f^2\,d\nu
=\int_0^s\int\Bigl[2f\varphi_K^U(L_{\mu}f-\widetilde{L}f)+
\\
f^2L_{\mu}\varphi_K^U
+2\langle A\nabla\varphi_K^U, \nabla f\rangle f
+|\sqrt{A}\nabla f|^2\varphi_K^U\Bigr]\,d\mu_t\,dt.
\end{multline*}
Note that
$$
L_{\mu}\varphi_K^U=K^{-1}\varphi'(U/K)L_{\mu}U+K^{-2}\varphi''(U/K)|\sqrt{A}\nabla U|^2.
$$
Since this expression doesn't equal zero only for $K\le U(x)\le 2K$, then
$$
|L_{\mu}\varphi_K^U|\le 2CI_K\Bigl(\frac{|L_{\mu}U|}{U}+\frac{|\sqrt{A}\nabla U|^2}{U^2}\Bigr),
$$
where $I_K$ is the indicator function of the set $\{x: \, K\le U(x)\le 2K\}$.
Similarly
$$
|\langle A\nabla\varphi_K^U, \nabla f\rangle f|\le 8C^2C^2(\psi)I_K\frac{|\sqrt{A}\nabla U|^2}{U^2}
+\frac{1}{4}\varphi_K^U|\sqrt{A}\nabla f|^2.
$$
Since
$$
|f\varphi_K^U(L_{\mu}-\widetilde{L})f|\le \varphi_K^U|f||A-A_m||D^2f|+
\varphi_K^U|f||\nabla f|\lambda^{-1}\bigl(|b_n|+|b|\bigr),
$$
the following bound holds:
\begin{multline*}
|f\varphi_K^U(\widetilde{L}f-L_{\mu}f)|
\le C(\psi)C(n,N,\psi)|A-A_m|+
\\
+16(\beta+1)V(x)e^{2(\alpha+1)(s-t)}+\frac{1}{4}|\sqrt{A}\nabla f|^2\varphi_{K}^U.
\end{multline*}
Gathering all bounds together one arrives at
$$
\int_0^s\int|\sqrt{A}\nabla f|^2\varphi_K^U\,d\mu_t\,dt\le C_1(1+R_{m}+Q_K),
$$
where
$$
R_{m}=C(n,N,\psi)\|A-A_m\|_{L^{1}(\mu+\sigma, B_{2N}^W\times[0, T])}, \quad
Q_K=C^2(\psi)\int_{K<U<2K} V\,d(\mu+\sigma)\Bigr)
$$
and  $C_1$ is independent of  $m, n, N, K, s$ and $\psi$.
A similar bound with  $\sigma$ instead of  $\mu$ holds. Indeed, we haven't used the fact that  $b_n$
 approximates  $b(\mu)$.
Now substitute  $u=f\varphi_K$ into identities (\ref{r2}) defining solutions $\mu$ and $\sigma$. Then
\begin{multline}\label{id1}
\int\varphi_K^U\psi\,d\mu_s-\int\varphi_K^Uf\,d\nu=\int_0^s\int\Bigl[\varphi_K(\widetilde{L}f-L_{\mu}f)+
\\
+fL_{\mu}\varphi_K^U+2\langle A\nabla\varphi_K^U, \nabla f\rangle\Bigr]\,d\mu_t\,dt,
\end{multline}
\begin{multline}\label{id2}
\int\varphi_K^U\psi\,d\sigma_s-\int\varphi_K^Uf\,d\nu=
\int_0^s\int\Bigl[\varphi_K(\widetilde{L}f-L_{\sigma}f)+
\\
+fL_{\sigma}\varphi_K^U+2\langle A\nabla\varphi_K^U, \nabla f\rangle\Bigr]\,d\sigma_t\,dt,
\end{multline}
Let us estimate individual terms in the right-hand side of  (\ref{id1}) and (\ref{id2}).
Since
$$
|L_{\mu}\varphi_K^U||f|+|L_{\sigma}\varphi_K^U||f|\le
2C(\psi)\Bigl(\frac{|L_{\mu}U|}{U}+\frac{|L_{\sigma}U|}{U}
+\frac{|\sqrt{A}\nabla U|^2}{U^2}\Bigr)\le 2(\beta+1)C(\psi)I_KV,
$$
one has
\begin{multline*}
\int_0^s\int|L_{\mu}\varphi_K^U||f|\,d\mu_t\,dt+\int_0^s\int|L_{\sigma}\varphi_K^U||f|\,d\sigma_t\,dt\le
\\
\le2(\beta+1)C(\psi)\int_0^T\int_{K\le V\le 2K}V\,d(\mu_t+\sigma_t)\,dt.
\end{multline*}
Due to Cauchy-Bunyakovsky inequality,
\begin{multline*}
\int_0^s\int|\sqrt{A}\nabla\varphi_K^U||\sqrt{A}\nabla f|\,d(\mu_t+\sigma_t)\,dt\le
\\
\le C\Bigl(\int_0^s|\sqrt{A}\nabla f|^2\varphi_K^U\,d(\mu_t+\sigma_t)\,dt\Bigr)^{1/2}
\Bigl(\int_0^T\int_{K<V<2K}\frac{|\sqrt{A}\nabla U|^2}{U^2}\,d(\mu_t+\sigma_t)\,dt\Bigr)^{1/2},
\end{multline*}
which is bounded by
$$
C_2\Bigl(1+R_{m}+Q_K\Bigr)^{1/2}
\Bigl(\int_0^T\int_{K<U<2K}V\,d(\mu_t+\sigma_t)\,dt\Bigr)^{1/2}.
$$
Here  $C_2$ does not depend on  $n, m, K, N, s, \psi$.
Since
$$
\varphi_K^U|\widetilde{L}f-L_{\mu}f|\le
C(n,N,\psi)|A-A_m|
+|\sqrt{A^{-1}}(b_n-b(\mu, \,\cdot\,,\,\cdot\,))||\sqrt{A}\nabla f|\varphi_K^U,
$$
the following estimate holds:
\begin{multline*}
\int_0^s\int\varphi_K^U|\widetilde{L}f-L_{\mu}f|\,d\mu_t\,dt\le
C(n,N,\psi)\|A-A_m\|_{L^{1}(\mu, B_{2N}^W\times[0, T])}+
\\
+C_1^{1/2}\|\sqrt{A^{-1}}(b_n-b(\mu, \,\cdot\,,\,\cdot\,))\|_{L^{2}(\mu, B_{2N}^W\times[0, T])}
\Bigl(1+R_{m}+Q_K\Bigr)^{1/2}
\end{multline*}
Finally, we have
$$
\varphi_K^U|\widetilde{L}f-L_{\sigma}f|\le \varphi_K^U|\widetilde{L}f-L_{\mu}f|
+\varphi_K^U|\sqrt{A^{-1}}(b(\mu, x, t)-b(\sigma, x, t))||\sqrt{A}\nabla f|.
$$
The first summand in the right-hand side of the last inequality is estimated as above. Consider the second summand.
Due to  (H3) and Cauchy-Bunyakovsky inequality,
\begin{multline*}
\int_0^s\int\varphi_K^U|\sqrt{A^{-1}}(b(\mu, x, t)-b(\sigma, x, t))||\sqrt{A}\nabla f|\,d\sigma_t\,dt\le
\\
\le\Bigl(\int_0^s\int G^2(\|\mu_t-\sigma_t\|_W)V\,d\sigma_t\,dt\Bigr)^{1/2}
\Bigl(\int_0^s\int|\sqrt{A}\nabla f|^2\varphi_K^U\,d\sigma_t\,dt\Bigr)^{1/2},
\end{multline*}
that is dominated by
$$
C_1^{1/2}\Bigl(1+R_{m}+Q_K\Bigr)^{1/2}
\Bigl(\int_0^sG^2(\|\mu_t-\sigma_t\|_W)\,dt\Bigr)^{1/2}.
$$
Subtracting (\ref{id2}) from (\ref{id1}) and applying all obtained estimates, at first  letting  $m\to\infty$,
then  $n\to\infty$ and finally  $K\to\infty$ (thus  $N\to\infty$ as well),
one gets
$$
\int\psi\,d(\mu_s-\sigma_s)\le C_1^{1/2}\Bigl(\int_0^sG^2(\|\sigma_t-\mu_t\|_{W})\,dt\Bigr)^{1/2}.
$$
Taking into account that  $\psi$ is an arbitrary function from $C^{\infty}_0(\mathbb{R}^d)$ such that $|\psi(x)|\le W(x)$,
 we obtain
$$
\|\mu_s-\sigma_s\|_{W}\le C_1^{1/2}\Bigl(\int_0^sG^2(\|\sigma_t-\mu_t\|_{W})\,dt\Bigr)^{1/2}.
$$
Gronwall's inequality yields  $\|\mu_s-\sigma_s\|_{W}=0$ for all  $s\in(0, T)$.
\end{proof}
\begin{remark}\rm
Previous theorem remains valid if one takes  $W\equiv 1$. The proof is much simplier in this case.
\end{remark}

\section{\sc Diffusion Matrix is independent of $\mu$ but can be degenerate}

If the diffusion matrix is degenerate, continuity of coefficients with respect to
 total variation of measure does not ensure uniqueness.
Indeed, let $A=0$ and
$$
b(\mu_t)=\int |y|^{2/3}\,d\mu_t.
$$
Then measure  $\delta_{x(t)}$ satisfies the equation   $\partial_t\mu={\rm div}(b(\mu_t)\mu)$
with initial data  $\mu|_{t=0}=\delta_0$
as soon as $x(t)$ satisfies the Cauchy problem $\dot{x}=|x|^{2/3}$, $x(0)=0$.
But the latter has two solutions: $x(t)=t^{3}/27$ and $x(t)=0$.
Hence one has to assume continuity with respect to some other
probability metric.

Suppose $W\in C(\mathbb{R}^d)$ and $W\ge 1$.
Set $\widetilde{W}(x)=\displaystyle\int_0^1\sqrt{W(tx)}\,dt$.
On the space of probability measures  $\mu$ satisfying
$|x|\widetilde{W}(x)\in L^1(\mu)$ we introduce a new metric
$$
w_{W}(\mu, \sigma)=\sup\Bigl\{\int f\,d(\mu-\sigma)\, : \, f\in C^{\infty}_0(\mathbb{R}^d), \, |\nabla f(x)|\le \sqrt{W(x)}\Bigr\}.
$$
If $W=1$ then  $w_W$ coincides with Kantorovich  1-metric $W_1(\mu, \sigma)$ (cf. \cite{Bogachev-2}).
In the general case $W_1(\mu, \sigma)\le w_{W}(\mu, \sigma)$.

In applications and principal examples $\sqrt{W}$ is often a convex function on $\mathbb{R}^d$.
Then the function $|x|\widetilde{W}(x)$ is integrable with respect to probability measure $\mu$ if
$|x|\sqrt{W(x)}$ is integrable with respect to $\mu$.
Moreover, metric  $w_W$ admits several equivalent definitions.
Define
$$
\mathcal{F}_0'=\Bigl\{f\in C^{\infty}_0(\mathbb{R}^d):\, |\nabla f(x)|\le \sqrt{W(x)}\Bigr\},
$$
$$
\mathcal{F}_0=\Bigl\{f\in C^{\infty}_0(\mathbb{R}^d):\, |f(x)-f(y)|\le|x-y|\max\{\sqrt{W(x)},\sqrt{W(y)}\}\Bigr\},
$$
$$
\mathcal{F}=\Bigl\{f\in C^{}(\mathbb{R}^d):\, |f(x)-f(y)|\le|x-y|\max\{\sqrt{W(x)},\sqrt{W(y)}\}\Bigr\}.
$$
Define
$$
d_{\mathcal{F}}(\mu, \sigma)=\sup_{f\in\mathcal{F}}\int f\,d(\mu-\sigma)
$$
and similarly  $d_{\mathcal{F}_0}$ and $d_{\mathcal{F}_0'}=w_{W}$ for
$\mathcal{F}_0$ and $\mathcal{F}_0'$ respectively.
Further it will be more convenient to use metric $w_W$, but
in applications it is often easier to check assumptions with
$d_{\mathcal{F}}$ or $d_{\mathcal{F}_0}$.

\begin{proposition}
Suppose $\sqrt{W}$ is a convex function on  $\mathbb{R}^d$ such that  $W\ge 1$.
Then metrices $d_{\mathcal{F}_0'}$, $d_{\mathcal{F}_0}$ and $d_{\mathcal{F}}$ coincide
on the set of measures $\mu$ with  $|x|\sqrt{W(x)}\in L^1(\mu)$.
\end{proposition}
\begin{proof}
The identity $d_{\mathcal{F}_0'}=d_{\mathcal{F}_0}$ follows from Newton-Leibnitz formula
$$f(x)-f(y)=\int_0^1\langle\nabla f(y+t(x-y)), x-y\rangle\,dt$$
 and the convexity of  $\sqrt{W}$.
Now we note that  $d_{\mathcal{F}_0}\le d_{\mathcal{F}}$. Let us prove the opposite inequality.
 Let $\mu, \sigma$ be probability measures satisfying
condition $|x|\sqrt{W(x)}\in L^1(\mu+\sigma)$.
For each $\varepsilon>0$ we find such $f\in \mathcal{F}$ that
\begin{equation}\label{eqqq}
d_{\mathcal{F}}(\mu, \sigma)\le \int f\,d(\mu-\sigma)+\varepsilon.
\end{equation}
Consider a cut-off function  $\psi_N(t)=t$ for $t\in[-N, N]$, $\psi_N(t)=N$ for $t>N$ and $\psi_N(t)=-N$ for $t<-N$.
Set $\varphi_K(x)=\varphi(x/K)$ where  $\varphi\in C^{\infty}_0(\mathbb{R}^d)$, $0\le\varphi\le 1$
and $\varphi(x)=1$ for $|x|\le 1$. For each $\delta\in(0, 1)$
set $g_{\delta, K, N}(x)=(1-\delta)\varphi_K(x)\psi_N(f(x))$. Note that for sufficiently small
$\delta$ and sufficiently large  $N$ and $K$ the function $f$ in (\ref{eqqq}) can be replaced with $g_{\delta, K, N}$ after taking
$2\varepsilon$ instead of $\varepsilon$ in the right-hand side. Since $g_{\delta, K, N}$ is
compactly supported and one can take sufficiently
large $K$, one has
$$
|g_{\delta, K, N}(x)-g_{\delta, K, N}(y)|\le (1-\delta/2)|x-y|\max\{\sqrt{W(x)}, \sqrt{W(y)}\}.
$$
By standard convolution with a smooth kernel one can smooth the function $g_{\delta, K, N}$,
and the coefficient $(1-\delta/2)$ ensures that the smoothed function belongs to  $\mathcal{F}_0$.
Thus for each~$\varepsilon>0$ one has
$d_{\mathcal{F}}(\mu, \sigma)\le d_{\mathcal{F_0}}(\mu, \sigma)+3\varepsilon$.
Hence
$d_{\mathcal{F}}(\mu, \sigma)\le d_{\mathcal{F_0}}(\mu, \sigma)$.
\end{proof}

Now let us consider even a more particular but important case: $W(x)=(1+|x|^{p-1})^2$ with $p\ge 2$.
Corresponding metric  $w_{W}$ is denoted by $w_p$.
Let us compare  $w_p$ with other probability  metrics (cf. \cite{Rachev91}):

\, 1) \, Fortet-Mourier metric
$$
t_p(\mu, \sigma)=\inf_Q\int_{\mathbb{R}^d_x\times\mathbb{R}^d_y}
|x-y|(1+\max\{|x|^{p-1}, |y|^{p-1}\})\,dQ
$$
where $Q$ is a finite Borel (possibly signed) measure on $\mathbb{R}^d_x\times\mathbb{R}^d_y$ with marginals
$Q_x$ on $\mathbb{R}^d_x$ and $Q_y$ on $\mathbb{R}^d_y$, such that $Q_x-Q_y=\mu-\sigma$;

\, 2) \, metric
$$
T_p(\mu, \sigma)=\inf_{P}\int_{\mathbb{R}^d_x\times\mathbb{R}^d_y}
|x-y|(1+\max\{|x|^{p-1}, |y|^{p-1}\})\,dP
$$
where  $P$ is a probability measure on $\mathbb{R}^d_x\times\mathbb{R}^d_y$ with marginals
$P_x=\mu$ on $\mathbb{R}^d_x$ and $P_y=\sigma$ on $\mathbb{R}^d_y$;

\, 3) \, Kantorovich $p$-metric
$$
W_p(\mu, \sigma)=\inf_P\Bigl(\int_{\mathbb{R}^d_x\times\mathbb{R}^d_y}
|x-y|^p\,dP\Bigr)^{1/p}
$$
where $P$ is a probability measure on  $\mathbb{R}^d_x\times\mathbb{R}^d_y$ with marginals
$P_x=\mu$ on $\mathbb{R}^d_x$ and $P_y=\sigma$ on $\mathbb{R}^d_y$.

\begin{proposition}
Suppose $|x|^p\in L^1(\mu+\sigma)$. Then

\, {\rm (i)} \, $t_p(\mu, \sigma)=w_p(\mu, \sigma)$,

\, {\rm (ii)} \, $t_p(\mu, \sigma)\le T_p(\mu, \sigma)\le 2p\, t_{p}(\mu, \sigma)$,

\, {\rm (iii)} \, $W_p(\mu, \sigma)\le 2T_p^{1/p}(\mu, \sigma)$ and
$$
T_p(\mu, \sigma)\le \Bigl(1+\int|x|^p\,d\mu+\int|x|^p\,d\sigma\Bigr)^{(p-1)/p}W_p(\mu, \sigma).
$$
\end{proposition}
\begin{proof}
Statement  (i) follows from  \cite[Theorem 5.3.2]{Rachev91}. Let us prove (ii).
First inequality  $t_p\le T_p$ is obvious. To prove the second one
 let us consider the following metric on  $\mathbb{R}^d$:
$$
d_p(x, y)=|x-y|+||x|^{p-1}x-|y|^{p-1}y|.
$$
Due to  \cite[Theorem 8.10.41]{ Bogachev-2} one has
$$
\inf_{P: P_x=\mu, P_y=\sigma}\int_{\mathbb{R}^d_x\times\mathbb{R}^d_y}d_p(x, y)\,dP=
\sup_{f: f(x)-f(y)\le d_p(x, y)}\int f\,d(\mu-\sigma).
$$
Due to  \cite[Lemma 4,5]{ZolotTVP76} the following inequalities hold true
$$
|x-y|(1+\max\{|x|^{p-1}, |y|^{p-1}\})\le 2d_p(x, y)\le 2p|x-y|(1+\max\{|x|^{p-1}, |y|^{p-1}\}).
$$
Hence
$$
T_p(\mu, \sigma)\le \inf_{P: P_x=\mu, P_y=\sigma}\int_{\mathbb{R}^d_x\times\mathbb{R}^d_y}2d_p(x, y)\,dP
$$
and
$$
\sup_{f: f(x)-f(y)\le 2d_p(x, y)}\int f\,d(\mu-\sigma)\le p w_p(\mu, \sigma)=2p t_p(\mu, \sigma).
$$
Note that in one dimentional case inequality $T_p\le 2p t_p$ is proved~in~\cite[Theorem 6.4.1]{Rachev91}.

Let us prove (iii). First inequality follows from  $$|x-y|^p\le 2^{p-1}|x-y|(|x|^{p-1}+|y|^{p-1}),$$
 second is ensured by  H\"older's inequality.
\end{proof}

Note that in typical cases coefficients are convolutions with polynomially growing kernels. Thus
 metric $T_p$ appears naturally in bounds for $|b(x, t, \mu)-b(x, t, \sigma)|$ and the latter metric
 can be estimated by $w_p$.

In the present paper we are interested in general (not only polynomial) the function~$W$.
For instance, the drift coefficient
$$
b(x, t, \mu)=\int_{\mathbb{R}^d}K(x, y, t)\,d\mu_t
$$
satisfies  $|b(x, t, \mu)-b(x, t, \sigma)|\le C(x, t)w_W(\mu_t, \sigma_t)$
where $\sqrt{W}$ is convex, if
$$
|K(x, y, t)-K(x, z, t)|\le C(x, t)|y-z|\max\{\sqrt{W(y)}, \sqrt{W(z)}\}.
$$
Thereby we can consider convolutions with kernels having not only polynomial, but arbitrary growth, determined
by the function  $\sqrt{W}$.

\begin{remark}\rm
To compare conditions ensuring uniqueness provided below, with
conditions from the existence result, it is userful to compare  $V$-convergence
and convergence in metric  $w_W$.
Suppose we have a $V$-convergent sequence of probability measures $\mu_n$ on $\mathbb{R}^d$
 with limit  $\mu$ and
$$
\sup_n\int V\,d\mu_n<\infty.
$$
If  $\lim_{|x|\to\infty}|x|\widetilde{W}(x)/V(x)=0$, then $\lim_{n\to\infty}w_W(\mu_n, \mu)=0$.

Let us prove it. As above, we denote over $\mathcal{F}$ the set of all functions
$f\in C^{\infty}_0(\mathbb{R}^d)$ such that
 $|\nabla f(x)|\le \sqrt{W(x)}$. Note that for all $f\in\mathcal{F}$
one has
$$
\lim_{n\to\infty}\int f\,d\mu_n=\int f\,d\mu, \quad
\int_{|x|>R} |f|\,d\mu_n\le
g(R)\int V\,d\mu_n, \quad g(R)=\sup_{|x|\ge R}\frac{|x|\widetilde{W}(x)}{V(x)}
$$
and $\lim_{R\to\infty}g(R)=0$.
Finally, due to Arzel\'a-Ascoli theorem the set  $\mathcal{F}$ on each ball $\{x: |x|\le R\}$
is a precompact set and thus has a finite $\varepsilon$-net for each $\varepsilon>0$.
Since convergence takes place for each element of this finite net and
any other function from~$\mathcal{F}$ can be uniformly on  $\{x: |x|\le R\}$ approximated by them,
 it yield (together with the uniform bound of integrals for $|x|\ge R$) the fact
 $\lim_{n\to\infty}w_W(\mu_n, \mu)=0$.

Thus, if  $\lim_{|x|\to\infty}|x|\widetilde{W}(x)/V(x)=0$, then the flow of probability measures  $\mu_t$
satisfying~(\ref{e1}) is continuous in  $t$ with respect to metric $w_W$.
\end{remark}

Let us remind that we consider only solutions from the class $\mathcal{M}_{T}(V)$ where
$V\in C(\mathbb{R}^d)$ and $V\ge 1$.
Let us state our assumptions on the coefficients.

\, (DH1) \, Matrix $A$ is symmetric and non-negative definite, $a^{ij}\in C(\mathbb{R}^d\times[0, T])$
and for each $t\in[0, T]$ the function  $x\mapsto a^{ij}(x, t)$ is twice continuously differentiable.
Let  $\sigma^{ij}$ denote the elements of the matrix $\sigma=\sqrt{A}$.

\, (DH2) \, For some function  $W\in C^2(\mathbb{R}^d)$ such that  $|x|\widetilde{W}(x)V(x)^{-1}$ is bounded on  $\mathbb{R}^d$,
$W\ge 1$ and
for each measure  $\mu\in \mathcal{M}_{T}(V)$ there exist functions $\theta_{\mu}, \Lambda_{\mu}\in C(\mathbb{R}^d)$
and such constants  $C_{\mu}>0$, $1>\delta_{\mu}>0$ that
$$
\langle b(x+y, t, \mu)-b(x, t, \mu), y\rangle\le\theta_{\mu}(x)|y|^2,
\quad L_{\mu}W(x, t)\le (C_{\mu}-\Lambda_{\mu}(x))W(x),
$$
\begin{multline}\label{ineq}
2\theta_{\mu}(x)+\delta_{\mu}(1+|x|^{2})^{-1}|b(\mu, x,t)|^{2}+
\\
+\delta_{\mu}(1+|x|^2)^{-2}|{\rm tr} A(x, t)|^2
+4\sum_{i,j,k\le d}\bigl|\partial_{x_{k}}\sigma^{ij}(x,t)\bigr|^{2}\le \Lambda_{\mu}(x)
\end{multline}
for all  $x, y\in\mathbb{R}^d$ and $t\in[0, T]$.

\, (DH3) \, For each ball $B\subset\mathbb{R}^d$  functions $b^i$ are continuous in  $x$ uniformly in  $t$ on $B\times[0, T]$
and there exists a continuous increasing function  $G$ on $[0, +\infty)$
that  $G(0)=0$ and
$$
|b(\mu, x, t)-b(\sigma, x, t)|\le V(x)W^{-1/2}(x)G(w_{W}(\mu_t, \sigma_t))
$$
for all  $(x, t)\in\mathbb{R}^d\times[0, T]$ and $\mu, \sigma\in \mathcal{M}_{T}(V)$.

\, (DH4) \, For some function  $U\in C^2(\mathbb{R}^d)$ satisfying
$U>0$ and $\lim\limits_{|x|\to+\infty}U(x)=+\infty$, and for each
measure $\mu\in\mathcal{M}_T(V)$ there is a constant  $\beta(\mu)$ such that
$$
\frac{|A(x, t)\nabla U(x)|\sqrt{W(x)}}{U(x)}+\frac{|\sqrt{A(x, t)}\nabla U(x)|^2}{U^2(x)}
+\frac{|L_{\mu}U(x, t)|}{U(x)}\le \beta(\mu)V(x)
$$
for all  $(x, t)\in\mathbb{R}^d\times[0, T]$.

\begin{theorem}\label{th2}
Suppose  {\rm (DH1)}, {\rm (DH2)}, {\rm (DH3)}, {\rm (DH4)} hold.
If
$$
\int_{0+}\frac{du}{G(u)}=+\infty,
$$
then there exists at most one solution to the Cauchy problem {\rm (\ref{e1})} from the class
$\mathcal{M}_{T}(V)$.
\end{theorem}

Before we provide the proof, let us consider an example.

\begin{example}\rm
Let  $m\ge 1$ and for each measure  $\mu$, given by a flow of probability measures~$(\mu_t)_{t\in[0, T]}$
on $\mathbb{R}^d$ such that
\begin{equation}\label{con2}
\sup_{t\in[0, T]}\int \exp(|x|^{2m})\,d\mu_t<\infty,
\end{equation}
there exist such constants  $c_1(\mu)>0$, $c_2(\mu)>0$, $c_3(\mu)>0$ that
$$
\langle b(\mu, x, t), x\rangle\le c_1(\mu)-c_2(\mu)|x|^2, \quad
|b(\mu, x, t)|\le c_3(\mu)\exp(|x|^{m}),
$$
$$
\langle b(\mu, x+y, t)-b(\mu, x, t), y\rangle\le c_3(\mu)(1+|x|^{m})|y|^2
$$
Suppose there exist a number  $c_4>0$ that
$$
|b(\mu, x, t)-b(\sigma, x, t)|\le c_4w_W(\mu_t, \sigma_t)\exp(|x|^{m}/2)
$$
for all  $\mu, \sigma$ satisfying condition (\ref{con2}). Here $W(x)=\exp(|x|^m)$.

Then the Cauchy problem
$$
\partial_t\mu+{\rm div}(b(\mu, x, t)\mu)=0, \quad \mu|_{t=0}=\nu,
$$
has at most one solution satisfying  (\ref{con2}).

For example, all assumptions are fulfilled for
$$
b(\mu, x, t)=-x\int\exp(|y|^2/3)\,d\mu_t.
$$
\end{example}

To prove the Theorem we need the following statement generalising a
result from  \cite{BDRSH15}.

Let $\eta\in C_{0}^{\infty}(\mathbb{R}^{1})$ be a cut-off function such that
$\eta(x)=1$ for $|x|\le 1$ and $\eta(x)=0$ for $|x|>2$, $0\le\eta\le1$
 and there exists a number  $C\ge 1$ that  $|\eta'(x)|^{2}\eta^{-1}(x)\le C$
for each  $x$ from the support of  $\eta$.

\begin{lemma}\label{lem-grad-est} Suppose $h^i, g^{ij}$ are continuous in  $(x, t)\in\mathbb{R}^{d+1}$
and twice continuously differentiable in $x$ functions such that matrix  $G=(g^{ij})$
 is non-negative definite.
Set  $Q=\sqrt{G}$ and
$
L_{g, h}u=g^{ij}\partial_{x_i}\partial_{x_j}u+h^i\partial_{x_i}u.
$
Suppose there is a continuous function  $\theta$
on $\mathbb{R}^{d}$ and numbers  $M>1$, $\delta>0$, $C_0>0$.
Set
$$\kappa=32^{-1}\min\{\delta, C^{-2}\}$$
 where $C$ is taken from the definition of  $\eta$.
If for all  $y\in\mathbb{R}^{d}$ and all $(x, t)$ such that  $|x|<(2M)^{\frac{1}{2\kappa}}$
 and $t\in[0, T]$
 one has
$$
\langle h(x+y,t)-h(x,t),y\rangle\le\theta(x)|y|^{2},
\quad
L_{g,h}W(x,t)\le(C_{0}-\Lambda(x,t))W(x),
$$
$$
\Lambda(x,t):=4\sum_{i,j,k\le d}\bigl|\partial_{x_{k}}q^{ij}(x,t)\bigr|^{2}+2\theta(x)+\delta(1+|x|^{2})^{-1}|h(x,t)|^{2}+
\delta(1+|x|^2)^{-2}|{\rm tr} G(x, t)|^2,
$$
 then for each  $s\in(0,T)$ the Cauchy problem
$$
\partial_{t}f+\zeta_{M}L_{g, h}f=0,\quad f|_{t=s}=\psi
$$
where $\psi\in C^{\infty}_0(\mathbb{R}^{d})$, $|\nabla\psi(x)|\le \sqrt{W(x)}$,
$\zeta_{M}(x)=\eta\bigl((1+|x|^{2})^{\kappa}/M\bigr)$, has a smooth solution ~$f$ and
$$
|\nabla f(x,t)|\le \sqrt{W(x)}\cdot e^{(C_{0}+1)(s-t)/2}.
$$
\end{lemma}
\begin{proof}
Existence of a smooth solution $f$ is well-known
(cf., for instance,  \cite[Theorem 2]{Ol}, \cite[Theorems 3.2.4,  3.2.6]{str}).
Let us obtain a bound for  $\nabla f$. Obviously the inequality
$\langle h(x+y,t)-h(x,t),y\rangle\le\theta(x)|y|^{2}$ for smooth functions $h$ yields
$$
\langle\mathcal{H}(x,t)y,y\rangle\le\theta(x)|y|^{2},\quad\mathcal{H}=(\partial_{x_{j}}h^{i})_{i,j\le d}.
$$
Moreover,  all inequalities from Lemma's formulation and the latter one should hold true only on the support of $\zeta_M$,
since the operator $L$ is multiplied by $\zeta_M$ and all coefficients are zero outside the support of $\zeta_M$.
Set  $u=2^{-1}\sum_{k=1}^{d}|\partial_{x_{k}}f|^{2}$. Differentiating the
equation $\partial_{t}f+\zeta_{M}L_{g, h}f=0$ with respect to $x_{k}$ and multiplying by
 $\partial_{x_{k}}f$ we obtain
\begin{multline*}
\partial_{t}u+\zeta_{M}L_{g, h}u+\zeta_{M}\langle\mathcal{H}\nabla f,\nabla f\rangle+\langle\nabla\zeta_{M},
\nabla f\rangle\langle h,\nabla f\rangle+\zeta_{M}\partial_{x_{k}}a^{ij}\partial_{x_{i}x_{j}}^{2}f\partial_{x_{k}}f+
\\
+g^{ij}\partial_{x_{i}x_{j}}^{2}f\partial_{x_{k}}f\partial_{x_{k}}\zeta_{M}-\zeta_{M}g^{ij}
\partial_{x_{k}x_{j}}^{2}f\partial_{x_{k}x_{j}}^{2}f=0.
\end{multline*}
Note that  $\langle\mathcal{H}\nabla f,\nabla f\rangle\le2\theta u$
and $\langle\nabla\zeta_{M},\nabla f\rangle\langle h,\nabla f\rangle\le2|\nabla\zeta_{M}||h|u$.
Consider the following expression:
$$
\zeta_{M}\partial_{x_{k}}g^{ij}\partial_{x_{i}x_{j}}^{2}f\partial_{x_{k}}f+g^{ij}\partial_{x_{i}x_{j}}^{2}
f\partial_{x_{k}}f\partial_{x_{k}}\zeta_{M}-\zeta_{M}g^{ij}\partial_{x_{k}x_{j}}^{2}f\partial_{x_{k}x_{j}}^{2}f.
$$
Remind that  $Q=\sqrt{G}$. Thus
\begin{multline*}
\sum_{i,j,k}\partial_{x_{k}}q^{ij}\partial_{x_{i}x_{j}}^{2}f\partial_{x_{k}}f=2\sum_{i,j,m,k}\partial_
{x_{k}}q^{im}q^{mj}\partial_{x_{i}x_{j}}^{2}f\partial_{x_{k}}f\le\\
\le2\sum_{i,m}\Bigl(\sum_{k}|\partial_{x_{k}}q^{im}|^{2}\Bigr)^{1/2}\Bigl(\sum_{k}|\partial_
{x_{k}}f|^{2}\Bigr)^{1/2}\Bigl|\sum_{j}q^{mj}\partial_{x_{i}x_{j}}^{2}f\Bigr|,
\end{multline*}
with is dominated by
$$
4u\sum_{i,m,k}|\partial_{x_{k}}q^{im}|^{2}+2^{-1}q_{i,m}\Bigl|\sum_{j}
q^{mj}\partial_{x_{i}x_{j}}^{2}f\Bigr|^{2}.
$$
Note that
$$
\sum_{i,m}\Bigl|\sum_{j}q^{mj}\partial_{x_{i}x_{j}}^{2}f\Bigr|^{2}=
\sum_{i,j,k}g^{ij}\partial_{x_{k}x_{j}}^{2}f\partial_{x_{k}x_{j}}^{2}f.
$$
Using an obvious inequality  $xy\le(4+4{\rm tr}G)^{-1}x^{2}+(1+{\rm tr}G)y^{2}$
we get
$$
g^{ij}\partial_{x_{i}x_{j}}^{2}f\partial_{x_{k}}f\partial_{x_{k}}\zeta_{M}\le2u\frac
{|\nabla\zeta_{M}|^{2}}{\zeta_{M}}(1+{\rm tr}G)+\zeta_{M}(4+4{\rm tr}G)^{-1}\Bigl(g^{ij}\partial_{x_{i}x_{j}}^{2}f\Bigr)^{2}.
$$
Note that
$$
\Bigl(\sum_{i,j=1}^{d}g^{ij}\partial_{x_{i}}\partial_{x_{j}}f\Bigr)^{2}
\le\Bigl(\sum_{i=1}^{d}g^{ii}\Bigr)\Bigl(\sum_{i,j,k}^{d}g^{ij}\partial_{x_{i}}
\partial_{x_{k}}f\partial_{x_{j}}\partial_{x_{k}}f\Bigr).
$$
This follows from  the inequality
$$|{\rm tr}(AB)|^2\le {\rm tr} A\,{\rm tr}(AB^2)$$
for symmetric matrices  $A$ and $B$ where
$A$ is non-negative. The latter can be derived by the application of Cauchy-Bunyakovsky inequality to
the scalar product
 $\langle X,Y\rangle={\rm tr}\,(XY^{*})$
of matrices $X=A^{1/2}$, $Y=BA^{1/2}$ in the space of  $d\times d$ matrices
(since  ${\rm tr}\,(YY^{*})={\rm tr}\,(BA^{1/2}A^{1/2}B)={\rm tr}\,(AB^{2})$).
Combining all estimates together, we get
$$
\partial_{t}u+\zeta_{M}L_{g,h}u+Zu\ge0,
$$
where
$$
Z=\frac{|\nabla\zeta_{M}|^{2}}{\zeta_{M}}(1+{\rm tr}G)+|\nabla\zeta_{M}|
|h|+2\zeta_{M}\theta+4\zeta_{M}\sum_{i,j,k}\bigl|\partial_{x_{k}}q^{ij}\bigr|^{2}.
$$
Since
$$
|\nabla\zeta_{M}(x)|\le4\kappa(1+|x|^{2})^{-1/2}
\bigl|\eta'\bigl((1+|x|^{2})^{\kappa}/M\bigr)\bigr|,
$$
we get
 $$
Z\le4\kappa C^2+16\kappa C
+\zeta_{M}\Bigl(4\sum_{i,j,k}
\bigl|\partial_{x_{k}}q^{ij}\bigr|^{2}+
2\theta+2\kappa(1+|x|^{2})^{-1}|h|^{2}+2\kappa(1+|x|^{2})^{-2}|{\rm tr}G|^{2}\Bigr).
$$
Choose such $\kappa>0$ that
$$
Z\le1+\zeta_{M}\Bigl(4\sum_{i,j,k}\bigl|\partial_{x_{k}}
\sigma_{N}^{ij}\bigr|^{2}+2\theta+\delta(1+|x|^{2})^{-1}|h|^{2}+\delta(1+|x|^{2})^{-2}|{\rm tr}G|^{2}\Bigr).
$$
Set  $u=wW$. Then  $w$ satisfies
$$
\partial_{t}w+\zeta_{M}L_{g,\widetilde{h}}w+\widetilde{Z}w\ge0,
$$
where
$$
\widetilde{h}^{k}=h^{k}+2\frac{g^{kj}\partial_{x_{j}}W}{W},\quad\widetilde{Z}=Z+\zeta_{M}\frac{L_{g,h}W}{W}.
$$
According to Lemma assumptions  $\widetilde{Z}\leq C_{0}+1$. Note that  $|w(x, s)|\le 1$.
Then the maximum pronciple (cf. \cite[Theorem 3.1.1]{str}) ensures  $|w(x, t)|\le e^{(C_0+1)(s-t)}$
which completes the proof.
\end{proof}

Now we can proceed to the proof of the Theorem.

\begin{proof}
Suppose there are two solutions $\mu$ and $\sigma$.
Let $\psi\in C^{\infty}_0(\mathbb{R}^d)$ and  $|\nabla\psi(x)|\le\sqrt{W(x)}$.
Set $M\ge 1$. Similarly to Lemma  \ref{lem-grad-est} set
$\kappa=32^{-1}\min\{\delta_{\mu}, C^{-2}\}$, where $C$ is taken from the definition of~$\eta$,
and  $\zeta_M=\eta((1+|x|^2)^{\kappa}/M)$. The function  $\eta$ is defined before Lemma~\ref{lem-grad-est}.
Let  $\varphi\in C^{\infty}_0(\mathbb{R})$ be a cut-off function such that  $0\le\varphi\le 1$,
$\varphi(x)=1$ for $|x|<1$ and
$\varphi(x)=0$ for $|x|>2$. Suppose also that for some  $C>0$ and all
 $x\in\mathbb{R}$ one has
$|\varphi''(x)|^2+|\varphi'(x)|^2\le C\varphi(x)$.
For each  $K\ge 1$ set
$\varphi^U_K(x)=\varphi(U(x)/K)$.
Assume  $M$ is big enough and  $\zeta_M(x)=1$ for  $|x|<2K$.
Set  $B_M=\{x: |x|<(2M)^{1/2\kappa}\}$.
Extend the functions  $b^i(\mu)$ on the whole space $\mathbb{R}^{d+1}$ as follows: $b^i(\mu, x, t)=b^i(\mu, x, T)$ for $t>T$ and
$b^i(\mu, x, t)=b^i(\mu, x, 0)$ for $t<0$. Extend  $a^{ij}$ in the similar way.
Obviously (DH1)-(DH4) are fulfilled for new  $b^i$ and $a^{ij}$.
According to Lemma \ref{approx} and Remark  \ref{approx-r} there exists a sequence
$b_n\in C^{\infty}(\mathbb{R}^{d+1})$ satisfying the following conditions:

\, (i) \, $\lim_{n\to\infty}\|b_n-b(\mu)\|_{L^1(\mu+\sigma, B_M\times[0, T])}=0$,

\, (ii) \, $\langle b_n(x+y, t)-b_n(x, t), y\rangle\le\widetilde{\theta}(x)|y|^2$
for all $(x, t)\in B_M\times[0, T]$ and $y\in\mathbb{R}^d$, where~$\widetilde{\theta}(x)=\theta_{\mu}(x)+1$,

\, (iii) \, for all $(x, t)\in B_M\times[0, T]$ one has  (\ref{ineq})
with $\widetilde{\theta}$ instead of $\theta_{\mu}$, $b_n$ instead of $b(\mu)$ and
$\widetilde{\Lambda}(x)=\Lambda_{\mu}(x)+2$ instead of $\Lambda_{\mu}$,

\, (iv) \, for all $(x, t)\in B_M\times[0, T]$ one has
$$
a^{ij}(x, t)\partial_{x_i}\partial_{x_j}W(x)+b^i_n(x, t)\partial_{x_i}W(x)\le
(\widetilde{C}_0-\widetilde{\Lambda}(x))W(x),
$$
where $\widetilde{C}_0=C_{\mu}+3$.

Suppose  $f_n$ satisfies the Cauchy problem
$$
\partial_tf_n+\zeta_Ma^{ij}\partial_{x_i}\partial_{x_j}f_n+\zeta_Mb^i_n\partial_{x_i}f_n=0, \quad f_n|_{t=s}=\psi.
$$
Due to the maximum principle  $\sup|f_n|=\max|\psi|$.
Lemma \ref{lem-grad-est} yields the following bound
$|\nabla_xf_n(x, t)|\le C_1\sqrt{W(x)}$, where $C_1$ is independent of  $x, t, s, n$ and $K$.
Substituting~$u=\varphi^U_Kf_n$ it into the definition identity~(\ref{r2})
for the solutions $\mu$ and $\sigma$,
we get
$$
\int\psi\,d\mu_s=\int\varphi^U_Kf_n\,d\nu+
\int_0^s\int\bigl[\varphi^U_K\langle b(\mu)-b_n, \nabla f_n\rangle
+2\langle A\nabla\varphi^U_K, \nabla f_n\rangle
+f_nL_{\mu}\varphi_k\bigr]\,d\mu_t\,dt,
$$
$$
\int\psi\,d\sigma_s=\int\varphi^U_Kf_n\,d\nu+
\int_0^s\int\bigl[\varphi^U_K\langle b(\sigma)-b_n, \nabla f_n\rangle+2\langle A\nabla\varphi^U_K, \nabla f_n\rangle
+f_nL_{\sigma}\varphi^U_K\bigr]\,d\sigma_t\,dt.
$$
Here we used the fact  $\zeta_M(x)=1$ for $x\in{\rm supp}\varphi_K^U$ and cancelled the terms
$\varphi_K^U\zeta_Ma^{ij}\partial_{x_i}\partial_{x_j}f_n$ and
$\varphi_K^Ua^{ij}\partial_{x_i}\partial_{x_j}f_n$.
Subtracting the second identity from the first, we come to
\begin{multline*}
\int\psi\,d(\mu_s-\sigma_s)\le
\int_0^s\int\bigl[\varphi^U_K|b(\mu)-b_n||\nabla f_n|+2|A\nabla\varphi^U_K||\nabla f_n|
+|f_n||L_{\mu}\varphi^U_K|\bigr]\,d\mu_t\,dt+
\\
+\int_0^s\int\bigl[\varphi^U_K|b(\sigma)-b_n||\nabla f_n|+2|A\nabla\varphi^U_K||\nabla f_n|
+|f_n||L_{\sigma}\varphi^U_K|\bigr]\,d\sigma_t\,dt.
\end{multline*}
Note that
$|b(\sigma)-b_n|\le |b(\sigma)-b(\mu)|+|b(\mu)-b_n|$ and $|\nabla f_n|\le C_1\sqrt{W}$.
Expressions  $|A\nabla\varphi^U_K|$, $|L_{\mu}\varphi^U_K|$ and $|L_{\sigma}\varphi^U_K|$ are estimated similarly
to the proof of Theorem  \ref{th1}.
Using~(DH4) and letting at first  $n\to\infty$ and then  $K\to\infty$, we come to
$$
\int\psi\,d(\mu_s-\sigma_s)\le C_1\int_0^s\int|b(\mu)-b(\sigma)|\sqrt{W}\,d\sigma_t\,dt.
$$
Using (DH3) and the definition of the metric  $w_W$, we arrive at
$$
w_W(\mu_s, \sigma_s)\le C_1N\int_0^sG(w_W(\mu_t, \sigma_t))\,dt, \quad N=\sup_t\int V\,d\sigma_t.
$$
Gronwall's inequality yields  $w_W(\mu_s, \sigma_s)=0$ for $s\in [0, T]$.
\end{proof}

\section{\sc Diffusion matrix depends on the solution}

Let us now consider the case when the diffusion matrix $A$ depends on $\mu$.
It is the most difficult situation as we need bounds for the second derivatives of the solution to the
adjoint problem. However, generally speaking one can not estimate them with first derivatives of the
initial condition.

Nevertheless, if the diffusion matrix is not degenerate, is bounded and is Lipschitz with respect to ~$x$,
one can estimate the second derivatives of the solution $f$ to the adjoint problem with first derivatives of  $f$
with a coefficient $(s-t)^{-1/2}$; the bound for the first derivatives can be obtained similarly to the previous
section.
So one can preserve the continuity assumption with respect to the metric $w_W$, introduced in the previous section.
This seems important as this metric arises naturally in most applied problems.
Other possible metrises are discussed at the end of the paper.

As above, we consider only solutions from the class $M_T(V)$ where $V\in C(\mathbb{R}^d)$ and~$V\ge 1$.
Let us introduce the following assumptions:

\, (NH1) \, For each  $\mu\in M_T(V)$ there exist constants $\lambda_{\mu}>0$ and $\Lambda_{\mu}>0$
such that
$$
\lambda_{\mu}^{-1}|\xi|^2\le \langle A(\mu, x, t)\xi, \xi\rangle\le \lambda_{\mu}|\xi|^2, \quad
|a^{ij}(\mu, x, t)-a^{ij}(\mu, y, t)|\le\Lambda_{\mu}|x-y|
$$
for all $x, y, \xi\in\mathbb{R}^d$, $t\in[0, T]$.

\, (NH2) \, For each  $\mu\in M_T(V)$ and each $x\in\mathbb{R}^d$ the following quantities are finite:
$$
B(\mu, x)=\sup_{t\in[0, T]}\sup_{|x-y|\le 1}|b(\mu, y, t)|,
$$
$$
\Theta(\mu, x)=\sup_{t\in[0, T]}\sup_{|x-y|\le 1, |x-z|\le 1, y\ne z}\frac{|b(\mu, y, t)-b(\mu, z, t)|}{|y-z|}.
$$
Moreover, for some function $W\in C^2(\mathbb{R}^d)$ such that
 $|x|\widetilde{W}(x)V(x)^{-1}$ is a bounded function, $W\ge 1$ and
for each measure $\mu\in \mathcal{M}_{T}(V)$ there exist such constants $C_{\mu}>0$, $1>\delta_{\mu}>0$
 that
$$
L_{\mu}W(x, t)\le (C_{\mu}-2\Theta(\mu, x)-\delta_{\mu}(1+|x|^2)^{-1}B^2(\mu, x))W(x).
$$

\, (NH3) \, There exists a continuous increasing function  $G$ on $[0, +\infty)$
such that  $G(0)=0$ and the following inequalities hold:
$$
|A(\mu, x, t)-A(\sigma, x, t)|\le G(w_{W}(\mu_t, \sigma_t)),
$$
$$
|b(\mu, x, t)-b(\sigma, x, t)|\le V(x)W^{-1/2}(x)G(w_{W}(\mu_t, \sigma_t))
$$
for all $(x, t)\in\mathbb{R}^d\times[0, T]$ and $\mu, \sigma\in \mathcal{M}_{T}(V)$.

\, (NH4) \,  There exists a function $U\in C^2(\mathbb{R}^d)$ such that
$U\ge 0$ and $\lim\limits_{|x|\to+\infty}U(x)=+\infty$, such that for each measure
 $\mu\in\mathcal{M}_T(V)$ there exists such a constant  $\beta(\mu)$ that
$$
\Bigl(B(\mu, x)+\sqrt{\Theta(\mu, x)}\Bigr)\sup_{|x-y|\le 1}\sqrt{W(y)}+
\frac{|\nabla U(x)|^2}{U^2(x)}
+\frac{|L_{\mu}U(x, t)|}{U(x)}\le \beta(\mu)V(x)
$$
for all  $(x, t)\in\mathbb{R}^d\times[0, T]$.

\begin{theorem}\label{th5}
Assume  {\rm (NH1)}, {\rm (NH2)}, {\rm (NH3)}, {\rm (NH4)} hold true.
If for some  $p>2$
$$
\int_0\frac{du}{G^p(u^{1/p})}=+\infty,
$$
then there exists at most one solution to the Cauchy problem {\rm (\ref{e1})} from the class
$\mathcal{M}_{T}(V)$.
\end{theorem}

Let us give an example of application of the last theorem.

\begin{example}\rm
Let $\alpha=(\alpha^{ij}(x, y))$ be a symmetric positive definite matrix and
$$\lambda^{-1}I\le \alpha(x, y)\le \lambda I$$
for some $\lambda>0$ and all $x$.
Moreover, $|\alpha(x, y)-\alpha(z, y)|\le\Lambda|x-y|$ for all $x, y, z\in\mathbb{R}^d$.
Set
$$
A(\mu, x)=\int \alpha(x, y)(1+|y|^m)\,d\mu_t(y)
$$
for $m\ge 1$.
Then there exists at most one solution to the Cauchy problem
$$
\partial_t\mu=\partial_{x_i}\partial_{x_j}(a^{ij}(\mu, x)\mu), \quad \mu|_{t=0}=\nu,
$$
from the class of measures  $\mu$ satisfying
$$
\sup_{t\in[0, T]}\int|y|^{m+1}\,d\mu_t(y)<\infty.
$$
\end{example}

To prove Theorem \ref{th5}, we need the following lemma, generalizing  \cite[Theorem 1]{Kerr}.

\begin{lemma}\label{lem-dr}
Assume that functions $q^{ij}$ and $h^i$ are smooth bounded with all derivatives
on~$U(x_0, 2)\times(-1, s)$.
Set
$$
B(x_0)=\sup_{t\in(-1, s)}\sup_{U(x_0, 1/2)}|h(x, t)|, \quad
\Theta(x_0)=\sup_{t\in(-1, s)}\sup_{U(x_0, 1/2)}|D_x h(x, t)|.
$$
Suppose that the matrix $Q=(q^{ij})$ is symmetric and satisfies
$$
\lambda^{-1}\le Q\le \lambda I, \quad |Q(x, t)-Q(y, t)|\le \Lambda|x-y|
$$
for all  $x, y\in U(x_0, 2)$, $t\in (-1, s)$ and some positive numbers $\lambda, \Lambda$.
Then the classical solution $f\in C^{2, 1}\bigl(U(x_0, 2)\times(-1, s)\bigr)$ to
$$\partial_tf+q^{ij}\partial_{x_i}\partial_{x_j}f+h^if=0$$
for each  $t_0\in(0, s)$  admits the bound
$$
|D^2f(x_0, t_0)|\le \frac{C\sqrt{s+1}\Bigl(B(x_0)+\sqrt{\Theta(x_0)}+1\Bigr)}{\sqrt{s-t_0}}\sup_{U(x_0, 1)\times(-1, s)}|Df|,
$$
here $C$ depends only on  $d, \lambda, \Lambda$.
\end{lemma}
\begin{proof}
For $0<\varepsilon<\min\{\sqrt{s-t_0}, 1\}$ set
$$
v(y, s)=f(x_0+\varepsilon y, t_0+\varepsilon^2\tau)
$$
where $y\in U(0, 1)$ and $\tau\in (-1, 1)$.
Notice that  $v_y=\varepsilon f_x$, $v_{yy}=\varepsilon^2f_{xx}$, $v_{\tau}=\varepsilon^2f_t$.
Substituting to  the equation and multyplying by  $\varepsilon^{-2}$, one gets
$$
v_{\tau}+\widetilde{q}^{ij}v_{y_iy_j}+\widetilde{h}^iv_{y_i}=0,
$$
where
$$
\widetilde{q}^{ij}=q^{ij}(x_0+\varepsilon y, t_0+\varepsilon^2 \tau), \quad
\widetilde{h}^i=\varepsilon h^i(x_0+\varepsilon y, t_0+\varepsilon^2 \tau).
$$
Suppose $\varepsilon |h|+\varepsilon^2|h_x|\le 1$ for $y\in U(0, 1)$ and $\tau\in [-1, 1]$.
Then by \cite[Theorem 1.]{Kerr}
one has the following bound:
$$
\sup_{U(0, 1/4)\times[-1/3, 1/3]}|D^2v|\le C\sup_{U(0, 1/2)\times[-1/2, 1/2]}|v|.
$$
Note that  $v(y, s)-v(0, 0)$ satisfies the  equation, the last bound is preserved after adding a constant to the solution.
Hence one can assume  $v(0, 0)=0$.
Due to  \cite[Theorem 2.13.]{Lb} (cf. also \cite[Corollary 2.14]{Lb} and remarks after it),
one has
$$
\sup_{s\in[-1/2, 1/2]}|v(0, s)|=\sup_{s\in[-1/2, 1/2]}|v(0, s)-v(0, 0)|\le
C\sup_{U(0, 1/2)\times[-1, 1]}|Dv|.
$$
Moreover,
$$
|v(y, s)-v(0, s)|\le \sup_{U(0, 1)\times[-1, 1]}|Dv|
$$
Thus the following estimate holds:
$$
\sup_{U(0, 1/4)\times[-1/3, 1/3]}|D^2v|\le C\sup_{U(0, 1)\times[-1, 1]}|Dv|
$$
with constant $C$ depending only on  $\lambda, \Lambda, d$. In coordinates  $x, t$ one gets
$$
|D^2f(x_0, t_0)|\le C\varepsilon^{-1}\sup_{U(x_0, \varepsilon)\times[t_0-\varepsilon^2, t_0+\varepsilon^2]}|Df|.
$$
Finally, choose $\varepsilon$ as follows:
$$
\varepsilon=2^{-1}(s+1)^{-1/2}\sqrt{(s-t_0)}(B(x_0)+\Theta(x_0)^{1/2}+1)^{-1}.
$$
This completes the proof.
\end{proof}

Let us proceed to the proof of Theorem \ref{th5}.

\begin{proof}
Assume that there are two different solutions $\mu$ and $\sigma$.
Let  $\psi\in C^{\infty}_0(\mathbb{R}^d)$ and $|\nabla\psi(x)|\le\sqrt{W(x)}$.
Let $M\ge 1$. As above in Lemma \ref{lem-grad-est}, set
$\kappa=32^{-1}\min\{\delta_{\mu}, C^{-2}\}$, where $C$ is taken from the definition of  $\eta$,
and  $\zeta_M=\eta((1+|x|^2)^{\kappa}/M)$. The cut-off function  $\eta$ is defined before Lemma \ref{lem-grad-est}.
Let $\varphi\in C^{\infty}_0(\mathbb{R})$, $0\le\varphi\le 1$, $\varphi(x)=1$ if $|x|<1$ and
$\varphi(x)=0$ if $|x|>2$.
Assume also that for some number  $C'>0$ and all
 $x\in\mathbb{R}$ one has
$|\varphi''(x)|^2+|\varphi'(x)|^2\le C'\varphi(x)$.
For each  $K\ge 1$ set
$\varphi^U_K(x)=\varphi(U(x)/K)$.
Consider  $M$ large enough and thus $\zeta_M(x)=1$
for $|x|<3K$. Set  $B_M=\{x: |x|<(2M)^{1/2\kappa}\}$.
Extend  $b^i$ on the whole space $\mathbb{R}^{d+1}$ as follows: $b^i(\mu, x, t)=b^i(\mu, x, T)$ if $t>T$ and
$b^i(\mu, x, t)=b^i(\mu, x, 0)$ if $t<0$. Extend  $a^{ij}$ in the same way.
Obviously (DH1)--(DH4) are fulfilled for new $b^i$, $a^{ij}$.

Due to Lemma  \ref{approx} and Remark  \ref{approx-r} there exist sequences
$b_n^i, a^{ij}_n\in C^{\infty}(\mathbb{R}^{d+1})$ and $a^{ij}_n$ such that

\, (i) \, $\lim_{n\to\infty}\bigl(\|b_n^i-b^i(\mu)\|_{L^1(\mu+\sigma, B_M\times[-1, T])}
+\|a^{ij}_n-a^{ij}(\mu)\|_{L^1(\mu+\sigma, B_M\times[0, T])}\bigr)=0$,

\, (ii) \, for all  $(x, t)\in B_M\times[-1, T]$ and $y\in\mathbb{R}^d$ one has
$$\sup_{|x-y|\le 1/2}|b^i_n(y, t)|\le B(\mu, x), \quad
\sup_{|x-y|\le 1/2}|D_x b_n(y, t)|\le \Theta(\mu, x)$$
and the matrix $A_n=(a^{ij}_n)$ satisfies (NH1) with the same $\lambda$ and $\Lambda$.

\, (iii) \, for all $(x, t)\in B_M\times[-1, T]$ one has
$$
a^{ij}_n(x, t)\partial_{x_i}\partial_{x_j}W(x)+b^i_n(x, t)\partial_{x_i}W(x)\le
(\widetilde{C_{\mu}}-\delta_{\mu}(1+|x|^2)^{-1}B(\mu, x)-\Theta(\mu, x))W(x)
$$
with $\widetilde{C_{\mu}}=C_{\mu}+1$.

Let  $f_n$ be the solution of the Cauchy problem
$$
\partial_tf_n+\zeta_Ma^{ij}\partial_{x_i}\partial_{x_j}f_n+\zeta_Mb^i_n\partial_{x_i}f_n=0, \quad f_n|_{t=s}=\psi.
$$
Due to maximum principle  $\sup|f_n|=\max|\psi|$. Moreover, due to Lemma \ref{lem-grad-est}
one can derive $|\nabla_xf_n(x, t)|\le C'\sqrt{W(x)}$, and for  $x$ from the support of  $\varphi^U_K$
Lemma \ref{lem-dr} ensures
$$
|D^2_xf_n(x, t)|\le C\sqrt{T}(s-t)^{-1/2}\Bigl(1+B(\mu, x)+\sqrt{\Theta(\mu, x)}\Bigr)\sup_{|x-y|\le 1}\sqrt{W(y)}.
$$
Constants  $C'$ and $C$ do not depend on  $n$, $t$, $s$ and $K$.
Substituting $u=\varphi^U_Kf_n$ into identities of the form  (\ref{r2}) defining solutions $\mu$ and $\sigma$, one gets
$$
\int\psi\,d\mu_s=\int\varphi^U_Kf_n\,d\nu+
\int_0^s\int\bigl[\varphi^U_K(L_{\mu}-L_n)f_n+2\langle A\nabla\varphi^U_K, \nabla f_n\rangle
+f_nL_{\mu}\varphi_k\bigr]\,d\mu_t\,dt,
$$
$$
\int\psi\,d\sigma_s=\int\varphi^U_Kf_n\,d\nu+
\int_0^s\int\bigl[\varphi^U_K(L_{\sigma}-L_{n})f_n+2\langle A\nabla\varphi^U_K, \nabla f_n\rangle
+f_nL_{\sigma}\varphi^U_K\bigr]\,d\sigma_t\,dt.
$$
Here we used  $\zeta_M(x)=1$ for $x\in{\rm supp}\, \varphi_K^U$.
Subtracting one identity from another, one obtains
\begin{multline*}
\int\psi\,d(\mu_s-\sigma_s)\le
\int_0^s\int\bigl[\varphi^U_K|b(\mu)-b_n||\nabla f_n|+
\\
+\varphi^U_K|A(\mu)-A_n||D^2f_n|+2|A\nabla\varphi^U_K||\nabla f_n|
+|f_n||L_{\mu}\varphi^U_K|\bigr]\,d\mu_t\,dt+
\\
+\int_0^s\int\bigl[\varphi^U_K|b(\sigma)-b_n||\nabla f_n|+
\\
+\varphi^U_K|A(\sigma)-A_n||D^2f_n|+2|A\nabla\varphi^U_K||\nabla f_n|
+|f_n||L_{\sigma}\varphi^U_K|\bigr]\,d\sigma_t\,dt.
\end{multline*}
Notice that
$$
|b(\sigma)-b_n|\le |b(\sigma)-b(\mu)|+|b(\mu)-b_n|, \quad
|A(\sigma)-A_n|\le |A(\sigma)-A(\mu)|+|A(\mu)-A_n|.
$$
Applying  (DH4) and letting  $n\to\infty$, then  $K\to\infty$, one comes to
\begin{multline*}
\int\psi\,d(\mu_s-\sigma_s)\le C(1+\sqrt{T})\int_0^s\int\Biggl[|b(\mu)-b(\sigma)|\sqrt{W}+
\\
+|A(\sigma)-A(\mu)|(s-t)^{-1/2}\Bigl(1+B(\mu, x)+\sqrt{\Theta(\mu, x)}\Bigr)
\sup_{|x-y|\le 1}\sqrt{W(y)}\Biggr]\,d\sigma_t\,dt.
\end{multline*}
Applying  (DH3) and definition of metric  $w_W$, one obtains
$$
w_W(\mu_s, \sigma_s)\le CN\int_0^sG(w_W(\mu_t, \sigma_t))(1+(s-t)^{-1/2})\,dt, \quad N=\sup_t\int V\,d\sigma_t.
$$
Take $p>2$ and $p'=p/(p-1)<2$. Applying  H\"older's inequality, one has
$$
w_W(\mu_s, \sigma_s)^p\le \widetilde{C}\int_0^s G^p(w_W(\mu_t, \sigma_t))\,dt.
$$
 Gronwall's inequality yields $w_W(\mu_s, \sigma_s)=0$ for $s\in [0, T]$.
\end{proof}

\begin{remark}\rm
In the present work we have studied only the case of the nondegenerate diffusion matrix depending on solution;
the major reason is that we want to deal with metric $w_W$, and consider coefficients continuous with respect to it.
In the case of a degenerate diffusion matrix one can consider a new metric
$$
d(\mu, \sigma)=\sup\Bigl\{\int f\,d(\mu-\sigma):\, f\in C^{\infty}_0(\mathbb{R}^d),
|D f(x)|\le 1, |D^2f(x)|\le 1\Bigr\}.
$$
Then usuing estimates from  \cite[Theorem 3.2.4]{str} and repeating the proof of Theorem ~\ref{th2} it is possible to prove
uniqueness of solution to the Cauchy problem in the case of smooth coefficients, bounded together with their derivatives.
However, using this metric $d(\mu, \sigma)$ we can consider only convolutions with twice continuously differentiable
kernels (with bounded derivatives) as coefficients. To work with unbounded kernels
it might be interesting to study uniqueness problems for coefficients that are continuous with respect to
Zolotarev's metric
$$
Z_p(\mu, \sigma)=\sup\Bigl\{\int f\,d(\mu-\sigma):\, f\in C^1(\mathbb{R}^d),
|\nabla f(x)-\nabla f(y)|\le |x-y|(1+|x|^{p}+|y|^{p})\Bigr\}.
$$
This case is especially important when diffusion is nontrivial.
Note that some properties of Zolotarev's metric can be found in  \cite{Zolot}.
We only note that the relations between this metric and metrics $W_p$ and $T_p$ is not particulary studied.
\end{remark}

\begin{remark}\rm
Suppose that under  (DH2) and (DH4) from Theorem  \ref{th2}
one can choose constants  $\delta_{\mu}$, $C_{\mu}$ and $\beta(\mu)$
independent of  $\mu$ from some class $M_{T, \alpha}(V)$;
here  $\alpha\in C^{+}([0, T])$.
We remind that the class $M_{T, \alpha}(V)$ consists of all measures $\mu$ given by such flows of probability measures $\mu_t$
that
$$
\int V\,d\mu_t\le\alpha(t).
$$
Suppose probability measures  $\nu_1$ and $\nu_2$ on $\mathbb{R}^d$ satisfy $V\in L^1(\nu_1+\nu_2)$.
Assume that  $\mu^1(dxdt)=\mu_t^1(dx)\,dt$ and $\mu^2(dxdt)=\mu_t^2(dx)\,dt$ solve the Cauchy problem
 (\ref{e1}) with initial values $\nu_1$ and $\nu_2$ respectively and belong to the class $M_{T, \alpha}(V)$.
If (DH1)--(DH4) are fulfilled, then repeating the proof of Theorem \ref{th2}, one can derive
$$
w_W(\mu_t^1, \mu_t^2)\le w_W(\nu_1, \nu_2)+C\int_0^tG(w_W(\mu_s^1, \mu_s^2))\,ds, \quad t\in[0, T].
$$
Gronwall's inequality yields
$$
w_W(\mu_t^1, \mu_t^2)\le F^{-1}\Bigl(F(w_W(\nu_1, \nu_2))-Ct\Bigr),
$$
where $F(v)=\displaystyle\int_v^1\frac{du}{G(u)}$ and $F^{-1}$ is an inverse function to $F$.
In particular, if $G(u)=u$ we come to the estimate
$$
w_W(\mu_t^1, \mu_t^2)\le w_W(\nu_1, \nu_2)e^{Ct}.
$$
Analogous estimates hold true under the assumptions of Theorems  \ref{th1} and \ref{th5}.
\end{remark}

\section{\sc Examples of nonuniqueness.}

Let us consider several cases when
degeneracy of the diffusion matrix $A$, depending only on $\mu_t$,
yields nonuniqueness of solutions to the corresponding Cauchy problem.

\begin{theorem}\label{th6_1}
Set  $A=a(\mu_t)I$, $b=0$ where $a(\sigma)$ is a nonnegative function on some subset of
probability measures that has a single zero at  $\nu$.

The problem  {\rm (\ref{e1})} has at least two solutions (one of which is a stationary solution)
in each neighbourhood of zero iff for sufficiently small $\varepsilon$ the following  integral converges:
$$\int_0^\varepsilon \frac {dt}{f(t)}<+\infty$$
where  $f(\beta)$ is the value of the functional  $a(\mu)$ at the measure with density $\Gamma(\beta,\cdot)*_x\nu$,
here $\Gamma$ is the fundamental solution of the heat operator.
\end{theorem}
\begin{proof}

Suppose there are two different solutions in the sense of the identity~(\ref{r1})
and one of them is a stationary one that identically equals $\nu$,
second doesn't equal  $\nu$ in some deleted neighbourhood of zero.

Obviously $a(\nu)=0$.  If  $a(\mu_t)=0$ in some neighbourhood of zero for a.e.  $t$,
then the measure $\mu_t$ is constant in this neighbourhood due to the equation.
Thus, without lack of generality, one can assume that  $a(\mu(t))>0$ for $t\in(0,t_0)$.
Suppose that the measure $\mu$ satisfies the problem (\ref{e1}) with coefficients as above,
and  $a(\mu_t)>0$ for $t>0$.

Set $g(t):=a(\mu_t)$ and define the function $\tau(t)$ such that $\tau'(t)=g(t),\quad \tau(0)=0,$
 that is a one-to-one correspondance of segments $[0,t]$ and $[0,\tau(t_0)]$.
 Notice that the measure
$\tilde\mu_\tau=\mu|_{t=t(\tau)}$ satisfies the problem
$\partial_\tau \mu=\Delta\mu, \quad \mu|_{t=0}=\nu$
in the sense of the identity~(\ref{r1}).

This problem has a unique solution given by
$\mu(\tau)=\Gamma(\tau,\cdot)*_x\nu$
where  $\Gamma(t,x)$ is a fundamental solution of the heat operator $\partial_t-\Delta$.

Let us go back to the functional $a(\mu)$. Denote $f(\beta)$ its value on the measure with density $\Gamma(\beta,\cdot)*_x\nu$
and notice that $f(0)=0$. The definition of  $\tau(t)$ ensures that this function solves
$\tau'=f(\tau)$ with initial condition $\tau(0)=0$.
Due to Osgood's criterion, if such  $\tau$ exists, then the integral
$\displaystyle\int\limits_0 \frac{dt}{f(t)}$
converges. Moreover, if this integral is finite, then one can find such function $\tau(t)$
 that  $\tau'=f(\tau)$. Then measure $\mu(t)=\Gamma(\tau(t),\cdot)*_x\nu$
 solves the problem ~(\ref{e1}).
\end{proof}

Let us show the possible application of our criterion.

\begin{example}\rm
Suppose  $d=1$ and consider $a(\mu)$ of the form
$$ a(\mu_t)=\left|\frac{\sqrt{\pi}}{2}\int |x| d\mu_t\right|^{2\alpha }$$
with  $\alpha>0.$

Then the problem  (\ref{e1}) with initial condition $\nu=\delta_0$,
where $\delta_0$ is a Dirac measure in zero, has at least two solutions
in the class of measures with $\displaystyle\int |x|d\mu<\infty$ for  $\alpha<1$.

Indeed, the functional $a$ at a measure with density $\Gamma(\beta,x)$ equals
$$f(\beta)=\left(\frac{\sqrt{\pi}}{2\sqrt {4\pi \beta}}\int|x|  e^{-\frac{x^2}{4\beta}}dx\right)^{2\alpha}=\left(\sqrt{\beta}\int_0^\infty e^{-\frac{x^2}{4\beta}} d\left(\frac{x^2}{4\beta}\right)\right)^{2\alpha}=\beta^\alpha$$
and the integral  $\displaystyle\int_0^\varepsilon \frac{dx}{f(x)}$ converges.

For $\alpha\ge 1$ the stationary solution is unique.

Let us estimate the difference $a(\mu)-a(\sigma)$ for $\alpha=\frac12$.
Taking  $\psi_n\in C^\infty_0(\mathbb R)$,
$|\nabla \psi_n|<1$ such that
$$\left|\int(\psi-|x|)d\sigma\right|+\left|\int(\psi-|x|)d\sigma\right|<1/n,$$
we get
$$\frac{2}{\sqrt{\pi}}(a(\mu)-a(\sigma))=\int |x| d(\mu-\sigma)
=\int \psi_n d(\mu-\sigma)+\int (|x|-\psi_n) d(\mu-\sigma)\leq W_1(\mu,\sigma)+\frac 1n.
$$

Letting $n\to\infty$, we arrive at  $|a(\mu)-a(\sigma)|\leq\frac{\sqrt{\pi}}{2} W_1(\mu,\sigma)$.

This example shows that one can not refuse of the condition (NH1) in Theorem (\ref{th5}).

\end{example}

One can save uniqueness if one imposes more
restrictive assumptions on the functional~$a(\mu)$.

\begin{example}\rm
Suppose $d=1$ and $$a(\mu_t)=\int K(x)d\mu_t$$
where $K(x)$ is a nonnegative function with two continuous uniformly bounded derivatives,
$a(\mu)=0$ only at $\mu=\nu$ and $|x|\in L^1(\nu)$.
Then the problem (\ref{e1}) has a unique solution.

Let us estimate
$$f(\beta)=\int\int K(x) \Gamma(\beta,x-y) \nu(dy)dx=\int K*\Gamma(\beta,\cdot)d\nu$$

Using the properties of the fundamental solution, one can get the bounds for the derivatives:
$$f'(\beta)=\int K*\partial_\beta\Gamma(\beta,\cdot)\nu=
\int K*\Delta\Gamma(\beta,\cdot)d\nu=\int \Delta K\Gamma(\beta,\cdot)d\nu\leq C
$$
for each $\beta$.
In this case $f(\beta)\leq C\beta$ and
$\displaystyle\int \frac{d\beta}{f(\beta)}\geq C\int \beta^{-1}d\beta=+\infty$,
which ensures uniqueness.
\end{example}

\begin{example}\rm
Generally speaking, $C^2$-smoothness of the kernel $K$ cannot be replaced with Holder continuity of the first derivatives.

Consider a functional  $a(\mu)$ of the form
$$ a(\mu_t)=\int |x|^{2\alpha} d\mu_t$$
 with  $\alpha<1,$ $\nu=\delta_0$.

Since
$$f(\beta)=\frac{1}{\sqrt {4\pi \beta}}\int|x|^{2\alpha}e^{-\frac{x^2}{4\beta}}dx=C\beta^\alpha,$$
the solution of  (\ref{e1}) is not unique.
\end{example}

Nonuniqueness of solutions is, generally speaking,
preserved after adding terms of the first order. This can be easily seen from the following
example:
\begin{example}\rm
Suppose the problem
$$\partial_t \mu=\Delta(a(\mu)\mu),\quad \mu|_{t=0}=\nu, $$
has at least two solutions.

Then there exists such a functional $b(\mu)$ satisfying Lipschitz condition with respect to Kantorovich 1-metric,
that is nonzero of a subset of probability measures with finite first moment $\int|x|d\mu$, such that
the corresponding problem
$$\partial_t \mu=\partial_x\partial_x(a(\mu)\mu)+\partial_x(b(\mu)\mu),\quad \mu|_{t=0}=\nu$$
also has at least two solutions.

We construct the functional $b(\mu)$ as follows. Let  $\mu$ and $\sigma$ be two different solutions of the initial problem.
Set
$$b(\nu)=\inf\Bigl\{W_1(\mu_t,\nu), W_1(\sigma_t,\nu), t>0\Bigr\}.$$
Since measures $\Gamma(\beta,x)dx$ do not form a dense set in the space of probability measures,
$b(\nu)$ does not an identically zero function. Obviously
$$|b(\mu)-b(\sigma)|\leq W_1(\mu,\sigma). $$
Moreover, $\mu$ and $\sigma$ solve the constructed Cauchy problem.
\end{example}

Nevertheless, in some cases adding  first derivatives ensures uniqueness. Let us provide an example of this phenomenon.

\begin{example}\rm
The Cauchy problem
$$\partial_t \mu=\partial_x^2(a(\mu)\mu)+\lambda\partial_x\mu,\quad \mu|_{t=0}=\delta_0, $$
where $d=1$, $\displaystyle a(\mu)=\int_\mathbb{R} |x|\mu_t(dx)$
has a unique solution for each  $\lambda\ne 0$.

Consider for simplicity  $\lambda=1$ (one can assure this by scaling).
The change of variables $y=x+t$ yields to the problem:
\begin{equation}\label{problem_a_1'}
\partial_t \mu=\partial_y^2(\widetilde{a}(\mu)\mu),\quad \mu|_{t=0}=\delta_0
\end{equation}
where $\displaystyle\widetilde{a}(t,\mu)=\int_\mathbb{R} |y-t|\mu_t(dy). $

Similarly to Theorem \ref{th6_1}, it is sufficient to show that there exists a unique
function~$\tau(t)$ satisfying
$$
\tau'=\frac{1}{\sqrt{4\pi \tau}}\int|y-t|e^{-\frac{y^2}{4\tau}}dy, \quad{\tau(0)=0}
$$
or equivalently
$\tau'=
2\frac{\sqrt\tau}{\sqrt{\pi}}e^{-\frac{t^2}{4\tau}}+2t\Phi(\frac{t^2}{4\tau})-t$
where $\Phi(x)=\pi^{-1/2}\int_{-\infty}^xe^{-y^2}dy.$
One can easily derive the bound $\tau\le Ct^2$ for some constant $C$.

Denoting  $\sqrt{\tau}=t\,g(t)$ for $t>0$, we arrive at
\begin{equation}\label{g_probl}
t\,g'=\frac{e^{-\frac{1}{4g^2}}}{\sqrt\pi}+\frac1{g}\Phi(\frac{1}{2g})-\frac{1}{2g}-g\equiv \mathcal{F}(g).
\end{equation}

Since $\mathcal{F}'(g)\leq {-1}$, the function  $\mathcal{F}(g)$
is monotone and decreasing on $g>0$ from $+\infty$ to $-\infty$ and, hence,
has a unique zero at  $g_0<1$, corresponding to an asymptotically stable solution $g=g_0$.
Moreover, each solution of this equation for $t\to 0$
 tends either to $g_0$ or to $\pm\infty$.
Taking into account that  $|g(0)|<\infty$, we come to uniqueness of the solution $g=g_0$ and of the corresponding solution $\mu=\Gamma(g_0t^2, x+t)dx$
 to the problem~(\ref{e1}).
\end{example}

\vskip 2. ex
\centerline{  \bf   Acknowledgements.  }

The authors are grateful to Prof. Vladimir I. Bogachev for
fruitful discussions and valuable remarks.

The work of O.A.Manita and S.V.Shaposhnikov was partially supported by
RFBR projects  12-01-33009, 14-01-00237. S.V.Shaposhnikov was partially supported by
the Simons Foundation and RFBR projects 14-01-91158, 14-01-90406-Ukr-f-a, 14-01-00736,
SFB 701 of the Bielefeld University;
M.S.Romanov was partially supported by the Government grant of the Russian Federation ''On measures designed to attract leading scientists
to Russian institutions of higher education"  No.
11.G34.31.0054, signed by the Ministry of Education and Science of the Russian
Federation, the leading scientist, and Lomonosov Moscow State University.

\end{document}